\documentclass[a4paper, twoside]{amsart}

\usepackage[english]{babel}
\usepackage[utf8]{inputenc}
\usepackage{amsmath}
\usepackage{amsthm}
\usepackage{amssymb}
\usepackage{graphicx}
\usepackage{enumerate}
\usepackage{tikz-cd}
\usepackage{fullpage}
\usepackage{bbm}
\usepackage{hyperref}
\usepackage{fancyhdr}
\usepackage[top=1.5in,bottom=1in,right=1in,left=1in,headheight=12pt]{geometry}
\usepackage{stmaryrd}
\usepackage{color,soul} 

\usepackage[backend=biber,style=alphabetic, maxnames=100]{biblatex}
\usepackage{csquotes}

\usepackage{mathtools}
\DeclarePairedDelimiter{\ceil}{\lceil}{\rceil}

\renewcommand{\O}{\mathcal{O}}

\newcommand{\QQ}{\mathbf{Q}}
\newcommand{\QQp}{\mathbf{Q}_p}
\newcommand{\ZZp}{\mathbf{Z}_p}
\newcommand{\QQpb}{\overline{\mathbf{Q}}_p}
\newcommand{\ZZ}{\mathbf{Z}}
\newcommand{\CC}{\mathbf{C}}
\newcommand{\RR}{\mathbf{R}}

\newcommand{\FFp}{\mathbf{F}_p}
\newcommand{\FFpb}{\overline{\mathbf{F}}_p}
\newcommand{\Kinf}{K_\infty}
\newcommand{\Minf}{M_\infty}
\newcommand{\Fil}{\mathrm{Fil}}
\newcommand{\Gal}{\mathrm{Gal}}
\newcommand{\Frob}{\mathrm{Frob}}
\newcommand{\Hom}{\mathrm{Hom}}
\newcommand{\gr}{\mathrm{gr}}
\newcommand{\GL}{\mathrm{GL}}

\newcommand\restr[2]{{
  \left.\kern-\nulldelimiterspace 
  #1 
  \vphantom{|}
  \right|_{#2} 
  }}
\newcommand{\ovr}[1]{\overline{#1}}
\newcommand{\undr}[1]{\underline{#1}}

\theoremstyle{definition}
\newtheorem{defn}{Definition}[section]

\theoremstyle{plain}
\newtheorem{thm}[defn]{Theorem}
\newtheorem{lem}[defn]{Lemma}
\newtheorem{prop}[defn]{Proposition}
\newtheorem{cor}[defn]{Corollary}

\newtheorem{thmx}{Theorem}

\theoremstyle{remark}
\newtheorem{rem}[defn]{Remark}

\begin{document}

\title{Explicit Serre weights for two-dimensional Galois representations over a ramified base}

\author{Misja F.A. Steinmetz}
\address{Mathematical Institute, Leiden University, Niels Bohrweg 1, 2333 CA, Leiden, the Netherlands}
\email{m.f.a.steinmetz@math.leidenuniv.nl}


\date{March 2022}

\dedicatory{In memory of Bas Edixhoven}


\begin{abstract}
Given a totally real number field $F$ and a mod $p$ Galois representation $\rho\colon G_F\to \mathrm{GL}_2(\bar{\mathbf{F}}_p)$, we propose an explicit definition of the set of Serre weights $W(\rho)$ attached to $\rho$. We prove that our explicit definition is equivalent to previous definitions available in the literature. As a consequence we obtain an explicit Serre's modularity conjecture for Hilbert modular forms over totally real number fields. Our work generalises previous work of Demb\'{e}l\'{e}--Diamond--Roberts and Calegari--Emerton--Gee--Mavrides which together give explicit and equivalent sets of weights when $p$ is unramified in $F$.
\end{abstract}

\maketitle

\section{Introduction}
Let $F$ be a totally real field, $p$ a prime and $\rho\colon G_F \to \GL_2(\FFpb)$ a continuous representation of the absolute Galois group of $F$. The weight part of Serre's modularity conjecture aims to predict the set of weights $W(\rho)$ of Hilbert modular forms such that the reduction of their associated Galois representation leads to $\rho$. A conjectural set of weights was first formulated in \cite{bdj10} under simplifying assumptions and in \cite{blgg13} in general. In \cite{gls15} it was proved that these conjectures give the correct sets of weights. All descriptions of $W(\rho)$ in these papers make use of subtle questions in integral $p$-adic Hodge theory, making it hard to do explicit computations in many cases. In this paper we give an alternative explicit definition of $W(\rho)$ only using local class field theory. We prove that our explicit definition of $W(\rho)$ is equivalent to earlier definitions.

Since the set of weights $W(\rho)$ decomposes as the tensor product of the weights associated to the local representations, it suffices only to consider the local representation $\rho|_{G_{F_\mathfrak{p}}}$ for a fixed prime $\mathfrak{p}\mid p.$ If $f_\mathfrak{p}$ denotes the residue field of $F_\mathfrak{p}$, then the set of weights $W(\rho|_{G_{F_\mathfrak{p}}})$ consists of Serre weights $V$ of $\GL_2(f_\mathfrak{p})$ (see Defn. \ref{defn:Serre-wt}). When $\rho|_{G_{F_\mathfrak{p}}}$ is irreducible, the set of associated Serre weights $W(\rho|_{G_{F_\mathfrak{p}}})$ is completely explicit already. Let us therefore consider the case when $\rho|_{G_{F_\mathfrak{p}}}$ is reducible. Then $\rho|_{G_{F_\mathfrak{p}}}$ is the extension of $\chi_2$ by $\chi_1$ for characters $\chi_1,\chi_2\colon G_{F_\mathfrak{p}} \to \FFpb^\times$. Therefore, $\rho$ defines a cocycle $c_\rho \in H^1(G_{F_\mathfrak{p}}, \FFpb(\chi_1\chi_2^{-1}))$. For a given Serre weight $V$, we define a \textit{distinguished subspace} $L_V\subseteq H^1(G_{F_\mathfrak{p}}, \FFpb(\chi_1\chi_2^{-1}))$ using $p$-adic Hodge theory (see Defn. \ref{defn:dist-subspc}). Traditionally, the set of Serre weights associated to $\rho|_{G_{F_\mathfrak{p}}}$ is defined as: $V\in W(\rho|_{G_{F_\mathfrak{p}}})$ if and only if $c_\rho \in L_V.$

In this paper we use local class field theory and the Artin--Hasse exponential to define a basis $\{c_\alpha \mid \alpha \in W\}$ of $H^1(G_{F_\mathfrak{p}}, \FFpb(\chi_1\chi_2^{-1}))$ explicitly (see Cor. \ref{cor:defn-basis}), where $W$ is an indexing set. We define a subset $J_V^\mathrm{AH}\subseteq W$ (see Defn. \ref{defn:J-V}) and, roughly speaking, we define $L_V^\mathrm{AH}:=\mathrm{span}(\{c_\alpha \mid \alpha \in J_V^\mathrm{AH}\})$ (see Defn. \ref{defn:L_V^AH}). The main theorem of this paper is as follows.

\begin{thmx} The following subspaces of $H^1(G_{F_\mathfrak{p}}, \FFpb(\chi_1\chi_2^{-1}))$ are equal:
\[
L_V = L_V^\mathrm{AH}.
\]
\end{thmx}

We emphasise that the definition of $L_V^\mathrm{AH}$ is in terms of explicitly defined basis elements, and no $p$-adic Hodge theory is needed for our definition. Suppose we define a set of Serre weights $W^\mathrm{AH}(\rho)$ via: $V\in W^\mathrm{AH}(\rho|_{G_{F_\mathfrak{p}}})$ if and only if $c_\rho \in L_V^\mathrm{AH}.$ The set of Serre weights $W^\mathrm{AH}(\rho)$ is now defined in terms of local class field theory. As a consequence we get an explicit definition of $W(\rho)$ for arbitrary totally real fields $F$ and representations $\rho:G_F\to \GL_2(\FFpb)$. In our result we have no restrictions on the ramification of $p$ in $F$. From the previous theorem together with the results of \cite{gls15}, we immediately obtain the following corollary.

\begin{thmx}
Let $F$ be a totally real field and $p>2$ a prime. Suppose $\rho:G_F\to \GL_2(\FFpb)$ is a continuous representation. Suppose that $\rho$ is modular, that $\rho|_{G_F(\zeta_p)}$ is irreducible and if $p = 5$, suppose further that the projective image of $\rho|_{G_F(\zeta_p)}$ is not isomorphic to $A_5.$

For each place $\mathfrak{p}\mid p$ of $F$ with residue field $f_\mathfrak{p}$, let $V_\mathfrak{p}$ be a Serre weight of $\GL_2(f_\mathfrak{p}).$ Then $\rho$ is modular of weight $\otimes_{\mathfrak{p}|p} V_\mathfrak{p}$ if and only if $V_\mathfrak{p} \in W^\mathrm{AH}(\rho|_{G_{F_\mathfrak{p}}})$ for all $\mathfrak{p}\mid p$.
\end{thmx}

We note that the condition $p>2$ and the so-called Taylor--Wiles condition in this statement is due to the use of modularity lifting theorems in \cite{gls15}, but is not needed for the local results of this paper.

The approach of redefining $L_V$ explicitly in terms of local class field theory and the Artin--Hasse exponential can first be found in \cite{ddr16} under the assumption that $p$ is unramified in $F$. Assuming unramifiedness the authors of that paper construct a basis $\{c_i\mid 0\le i <f\}$ of $H^1(G_{F_\mathfrak{p}}, \FFpb(\chi_1\chi_2^{-1}))$ and an indexing set $\mu(J)$. Roughly speaking, they conjecture that $\mathrm{span}(\{c_i\mid i \in \mu(J)\}) = L_V.$ This `unramified conjecture' was later proved in \cite{cegm17}. In the unramified case it follows straightforwardly that the bases appearing in our paper and \cite{ddr16} coincide (assuming the same choices were made when defining the bases). It also follows, although less straightforwardly, that the indexing sets  $J_V^\mathrm{AH}$ and $\mu(J)$ coincide. In this paper we may therefore say that we generalise the results of \cite{ddr16} and \cite{cegm17} to allow for arbitrary ramification of $p$ in $F$.

The existence of an explicit Serre weights conjecture when $p$ is ramified in the base field is of great importance for many applications. It may come as no surprise that an explicit recipe for the set of weights may be used to compute the set of weights explicitly. An example of this approach can be found in \cite[\S8]{ddr16} in which the explicit description is used to calculate $W(\rho)$ in some non-trivial cases. One immediate application of the results in this paper in forthcoming work is the strengthening of \cite[Thm.~4.9]{ds15}. This theorem is the main local result used to prove \textit{weight elimination} (i.e. the set of weights is a subset of the set of predicted weights) for generic representations in \cite{ds15}. We can use the explicit results of this paper to show this theorem holds for a greater class of representations while at the same time largely simplifying the proof of the theorem. We have reasons to expect our class of representations to be optimal for this theorem to hold.

\subsection{Brief history and outline of the paper}
For a totally real field $F$, a prime $p$ and a mod $p$ Galois representation $\rho:G_F\to \GL_2(\FFpb)$, the question whether this representation can arise as the reduction of the Galois representation associated to a Hilbert modular form was first rigorously studied in \cite{bdj10} under the hypothesis that $p$ is unramified in $F$. This being a generalisation of Serre's famous conjecture \cite{ser87} for $F=\QQ$, it was long believed that a similar conjecture could be formulated for totally real fields. In the conjecture of \cite{bdj10}, the authors were able to predict the set of weights of Hilbert modular forms leading to the given representation precisely for the first time. The unramifiedness hypothesis was later removed in \cite{sch08}, \cite{gee11} and \cite{blgg13}. Building on previous work of Gee and co-authors it was proved in \cite{gls15} for $p>2$ that the predicted set of weights of \cite{bdj10} and its extensions is correct in the sense that if $\rho$ comes from a Hilbert modular form then it must come from a form of weight predicted by the conjectures and every predicted weight occurs. In \cite{wan17} the local results of \cite{gls15} and their proof of the \textit{weight elimination} part of the conjecture were extended to $p=2$. One unifying aspect of all these conjectures is that the predicted sets of weights are defined using abstract $p$-adic Hodge theory making it complicated to do explicit computations in many cases. The main goal of this paper is to give a more explicit definition of the set of weights attached to $\rho$.

In \S\ref{sec:background-notation} we will recall the usual definition of the set of weights $W(\rho)$ as found in the literature. Many equivalent definitions exist, however none of them explicit in the sense of the definition later introduced in this paper, so we pick a definition that is suitable to us.

In \S\ref{chap:expl-basis} we define an explicit basis of $H^1(G_K, \FFpb(\chi))$ using local class field theory and the Artin--Hasse exponential. We do this by first defining a filtration on $H^1(G_K,\FFpb(\chi))$ given in terms of restrictions to higher ramification groups. We study the jumps in this filtration. Then we use local class field theory to relate this filtration to a filtration on the units of a local field $M$ by higher unit groups. The Artin--Hasse exponential proves to be the perfect tool to get good control of elements in these higher unit groups increasingly closer to 1. Therefore, we can write down a basis of each graded piece for the filtration by unit groups and the `duality' with $H^1$ allows us to `translate' this back into a basis of $H^1(G_K,\FFpb(\chi)).$

In \S\ref{chap:proof} we first use the classification theorem of \cite{gls15} to give the distinguished spaces $L_V$ (after restriction to $H^1(G_{K_\infty},\FFpb(\chi))$) in terms of \'etale-$\varphi$ modules. Next we use Artin--Schreier theory and restriction to the absolute Galois group of a local field $M_\infty$ over which the character $\chi$ becomes trivial to describe $L_V$ as a subspace of $\Hom(G_{M_\infty},\FFpb).$ The basis from \S\ref{chap:expl-basis} is easily expressed in terms of $G_{M_\infty}$ using local class field theory because it was defined in terms of higher unit groups of $M$. Then we use an explicitly known formula for the evaluation pairing $\Hom(G_{M_\infty},\FFpb)\times G_{M_\infty}\to \FFpb$ to compare the basis elements of \S\ref{chap:expl-basis} to $L_V$.

Our explicit definition of the space $L_V$ is given in Definition \ref{defn:L_V^AH}. In this definition some unfamiliar quantities appear such as an indexing set $J_V^\mathrm{AH}$ defined in Definition \ref{defn:J-V}. All these quantities can be computed easily and quickly on a computer using any programming language that supports basic integer operations. This makes our reformulation of the conjecture really explicit. Nonetheless, it would be satisfying to find a more direct formula for expressing the indexing set $J_V^\mathrm{AH}$ as was done in \cite[\S7]{ddr16} under the assumption that $p$ is unramified in $F$. Some progress under other simplifying assumptions is made towards finding a direct formula for $J_V^\mathrm{AH}$ in \cite[Ch.~7]{ste20}, but in general the problem is still open.

We remark that special cases of \S\ref{chap:expl-basis} and \S\ref{chap:proof} are related to Abrashkin’s papers \cite{abr89} and \cite{abr97}.

\subsection{Acknowledgements}

I would like to thank Fred Diamond for suggesting this problem and his guidance over the years. I would also like to thank Victor Abrashkin, Robin Bartlett, Laurent Berger, Martin Bright, Dougal Davis, Toby Gee, Pol van Hoften, Ashwin Iyengar, Robin de Jong, James Newton, Esther Polak and Jan Vonk for helpful mathematical conversations and suggestions. This work was supported by the Engineering and Physical Sciences Research Council [EP/L015234/1]. During the research the author was a registered student at The EPSRC Centre for Doctoral Training in Geometry and Number Theory (The London School of Geometry and Number Theory), University College London.

\section{Background and notation}\label{sec:background-notation}

Let $F$ be a totally real field and $p$ a fixed prime. Given a representation $\rho: G_F\to \GL_2(\FFpb)$, in this section we will carefully define the set of Serre weights $W(\rho)$ attached to $\rho$ using $p$-adic Hodge theory as found in the literature. The definition we give here is \cite[Defn.~4.1.4]{blgg13}; we refer the reader to their paper and \cite{gls15} for a discussion of different possible definitions of the sets of weights attached to $\rho$. 

One of the important aspects of Serre's conjecture for a totally real field $F$ is that the set of weights $W(\rho)$ associated to this representation only depends on the restriction of $\rho$ to local representations $\restr{\rho}{G_{F_\mathfrak{p}}}:G_{F_\mathfrak{p}}\to\GL_2(\FFpb)$, where $\mathfrak{p}\mid p$ is a prime of $F$ dividing $p$ and $F_\mathfrak{p}$ denotes the completion of $F$ at this prime. This allows us to write the whole paper locally, because we get back to the global case via
\[
W(\rho):=\left\{V=\otimes_{\mathfrak{p}\mid p} V_\mathfrak{p}\;\middle |\; V_\mathfrak{p}\in W(\restr{\rho}{G_{F_\mathfrak{p}}})\text{ for all }\mathfrak{p}\mid p\right\}.
\] Therefore, we assume from now on $\rho:G_K\to\GL_2(\FFpb),$ where $K$ is a finite extension of $\QQp$. Another important feature is that the recipe for the set of weights $W(\rho)$ associated to $\rho$ is already completely explicit in the case that $\rho$ is an irreducible representation. Since our goal is to obtain an explicit way of describing this set, we may therefore assume that $\rho$ is reducible.

Since $\rho$ is reducible, we may write
\[
\rho\sim \begin{pmatrix}\chi_1 & * \\ 0 & \chi_2\end{pmatrix}=\chi_2\otimes\begin{pmatrix}\chi & c_\rho \\ 0 & 1\end{pmatrix}
\]
for some characters $\chi_1,\chi_2:G_K\to\FFpb^\times$ and where $\chi:=\chi_1\chi_2^{-1}.$ It is not hard to check that $c_\rho$ defines a cocycle in $H^1(G_K,\FFpb(\chi)),$ which is well-defined up to multiplication by a scalar in $\FFpb^\times.$ We will now define a subspace $L_V\subseteq H^1(G_K,\FFpb(\chi))$ depending on a so-called \emph{Serre weight} $V$. The set of weights $W(\rho)$ associated to $\rho$ is then defined via
\[
V\in W(\rho) \iff c_\rho \in L_V.
\]
In other words, an explicit description of the set of weights $W(\rho)$ is equivalent to an explicit description of the subspace $L_V$ for all Serre weights $V$. In the remainder of the paper we will try to do the latter.

\begin{defn}\label{defn:Serre-wt} A \textbf{Serre weight} for $K$ is an irreducible $\FFpb$-representation of $\GL_2(k)$, which is necessarily of the form
\[
V_{\underline{\eta},\underline{\theta}}:=\bigotimes_{\lambda \in \Hom(k,\FFpb)} (\det{}^{\theta_\lambda}\otimes_{k}\mathrm{Sym}^{\eta_\lambda-\theta_\lambda}k^2)\otimes_{k,\lambda}\FFpb
\]
for some uniquely determined integers $\eta_\lambda,\theta_\lambda$ with $\theta_\lambda,\eta_\lambda-\theta_\lambda\in[0,p-1]$ for all $\lambda$ and $\theta_\lambda<p-1$ for at least one $\lambda.$ 
\end{defn}

For a crystalline representation $\widetilde{\rho}: G_K\to \GL_2(\QQpb)$ and $\tau\in\Hom(K,\QQpb)$, we let $\mathrm{HT}_\tau(\widetilde{\rho})$ denote the $\tau$-labelled \textbf{Hodge--Tate weights} -- see \S\ref{subsec:notation} for our conventions on Hodge--Tate weights. Furthermore, we say that $\widetilde{\rho}$ has \textbf{Hodge type} $(\undr\eta,\undr\theta)$ if $\mathrm{HT}_{\tau_{i,0}}(\widetilde{\rho})=\{\theta_i,\eta_i+1\}$ and $\mathrm{HT}_{\tau_{i,j}}=\{0,1\}$ for $j>0.$ Following \cite{gls15} we will say a representation is \textbf{pseudo-Barsotti--Tate} (or pseudo-BT) of weight $\{\eta_i+1\}$ if it has Hodge type $(\eta,\undr 0).$ We will often write $\{r_i\}$ where $r_i:=\eta_i+1$ for the weight of a pseudo-BT representation.

\begin{defn}\label{defn:dist-subspc} For any ordered pair of characters $\chi_1,\chi_2:G_K\to\FFpb^\times$ and a Serre weight $V_{\undr\eta,\undr\theta}$, let the \textbf{distinguished subspace} $L_{V_{\undr\eta,\undr\theta}}(\chi_1,\chi_2)$ be defined as the subset of $H^1(G_K,\FFpb(\chi))$ obtained as the reduction of all crystalline representations of Hodge type $(\undr\eta,\undr\theta)$ of the form
\[
\begin{pmatrix}\widetilde{\chi}_1 & * \\ 0 & \widetilde{\chi}_2\end{pmatrix},
\]
where $\widetilde{\chi}_1$ and $\widetilde{\chi}_2$ are any crystalline lifts of $\chi_1$ and $\chi_2,$ respectively. 
\end{defn}
Note that any such reduction gives a representation which is an extension of $\chi_2$ by $\chi_1$ and, therefore, by the process explained above a class in $H^1(G_K,\FFpb(\chi))$. When $L_{V_{\undr\eta,\undr\theta}}(\chi_1,\chi_2)$ is non-empty, it is a subspace -- this follows, for example, from the results of \cite{gls15}.

It follows from the definitions that
\[
L_{V_{\undr\eta,\undr\theta}}(\chi_1,\chi_2)=L_{V_{\undr\eta-\undr\theta,\undr 0}}(\prod_i \omega_i^{-\theta_i}\otimes\chi_1,\prod_i \omega_i^{-\theta_i}\otimes\chi_2).
\]
Therefore it suffices to give an explicit description of the distinguished subspaces (for aribtrary $\chi_1$ and $\chi_2$) in the case when the determinant in our Serre weight is trivial, or, equivalently, we only consider reductions of pseudo-Barsotti--Tate representations in the definition of the distinguished subspaces.

\subsection{Notation}\label{subsec:notation}

In this paper $p$ will denote a fixed prime number and $K$ will be a finite extension of $\QQp$ with ramification index $e$, residue degree $f$ and residue field $k$. We will fix all necessary algebraic closures and embeddings between them; $G_K:=\Gal(\overline{K}/K)$ denotes the absolute Galois group of $K$ and similarly for other fields. We fix an embedding $\overline{\tau}_0\in \Hom(k,\FFpb)$ and label the remaining embeddings of the residue field via the rule $\overline{\tau}_{i+1}^p=\overline{\tau}_i.$ This is the convention as in \cite{gls15} and \cite{cegm17} and opposite to \cite{ddr16}. We now label the embeddings
\[
\Hom(K,\QQpb)=\{\tau_{i,j}\mid i=0,\dotsc,f-1;\:j=0,\dotsc,e-1\}
\]
in any way such that the embedding $k\to \FFpb$ induced by $\tau_{i,j}$ is $\overline{\tau}_i.$ We fix a uniformiser $\pi_K$ of $K$ and set $\pi_0:=\pi_K.$ For each $n\ge 1$ we choose a $p^n$-th root of $\pi_K$ such that $\pi_{n+1}^p=\pi_n.$ We define the rising union $K_\infty:=\bigcup_{n\ge 0} K(\pi_n);$ when $p=2$ we, furthermore, require that $\pi_K$ was chosen so that $K_\infty\cap K_{p^\infty}=K$, where $K_{p^\infty}:=\bigcup_{n\ge 1} K(\zeta_{p^n})$ with $\zeta_{p^n}$ a primitive $p^n$-th root of unity in $\overline{K}$ -- it follows from \cite[Lem.~2.1]{wan17} that this is always possible. For any unramified extension $L$ of $K$ we let $\Frob_K\in\Gal(L/K)$ denote the arithmetic Frobenius.

Let $\pi$ be a root of $x^{p^f-1}+\pi_K$ in $\overline{K}$. Recall that we have a character $\overline{\omega}_\pi: G_K\to k^\times$ defined by $\sigma \mapsto \sigma(\pi)/\pi \bmod{\pi_K}.$ For an embedding $\overline{\tau}\in \Hom(k,\FFpb)$, we get a character $\omega_{\overline{\tau}}:G_K\to \FFpb^\times$ via $\omega_{\overline{\tau}}:=\overline{\tau}\circ \overline{\omega}_\pi.$ For $\overline{\tau}=\overline{\tau}_i$ we will sometimes simply denote this character by $\omega_i;$ these characters are often called the \textbf{fundamental characters} of level $f$. The ($p$-adic) \textbf{cyclotomic character} $\epsilon:G_K\to \ZZp^\times$ of $K$ is defined by $\sigma(\zeta)=\zeta^{\epsilon(\sigma)}$ for any $\sigma\in G_K$ and $\zeta$ any compatible sequence of higher $p$-th roots of unity in $\mu_{p^\infty}(\overline{K})$; by abuse of notation for its reduction modulo $p$ we will sometimes use the same terminology: $\epsilon:G_K\to \FFp^\times.$

Let $\CC_K$ denote the completion of $\overline{K}.$ For an integer $i,$ we write $\CC_K(i)$ for the $i$-th \textbf{Tate twist}, i.e. the same underlying vector space with $G_K$ acting as the $i$-th power of the $p$-adic cyclotomic character. Given a crystalline representation $\rho:G_{K}\to \GL(V)$ on a $\QQpb$-vector space $V$ we define its Hodge--Tate weights as follows: for $\tau\in\Hom(K,\QQpb)$ we define the $\tau$-labelled \textbf{Hodge--Tate weights} $\mathrm{HT}_\tau(\rho)$ of $\rho$ to be the multiset of integers containing the integer $i$ with multiplicity
\[
\dim_{\QQpb}(V\otimes_{\tau, K} \CC_K(-i))^{G_K},
\]
where Galois acts diagonally. It follows from the definition that the cyclotomic character $\epsilon$ has Hodge--Tate weight 1.

\section{Explicit bases of spaces of extensions}\label{chap:expl-basis}

In this section we define an explicit basis of $H^1(G_K,\FFpb(\chi))$ in terms of local class field theory and the Artin--Hasse exponential. In the next section we will use this basis to give an explicit description of the distinguished subspaces $L_V$ from \S\ref{sec:background-notation}. Let us emphasise that the basis elements in this section may depend on the choice of uniformiser we make, but in the next section we will take an appropriate subset of basis elements spanning a subspace which is independent of this choice. The results of this section can be seen as generalising results of \S3 and \S5 of Demb\'el\'e, Diamond and Roberts \cite{ddr16}.

\subsection{The filtration on \texorpdfstring{$H^1(G_K,\FFpb(\chi))$}{H1(GK,FFpb(chi))}}\label{sec:filtration}

Let $\pi$ be a root of $x^{p^f-1}+\pi_K$ in $\overline{K}.$ Recall that we have a character $\overline{\omega}_\pi:G_K\to k^\times$ defined by $\sigma \mapsto \sigma(\pi)/\pi.$ For $i=0,\dotsc,f-1,$ we get the fundamental characters $\omega_i:G_K\to\FFpb^\times$ of level $f$ by defining $\omega_i:=\overline{\tau}_i\circ \overline{\omega}_\pi.$ If $\chi:G_K\to\FFpb^\times$ is any continuous character, then we can write $\restr{\chi}{I_K}=\prod_i\restr{\omega_i^{a_i}}{I_K}$ for $a_i\in [1,p].$\footnote{Our choice of $a_i\in[1,p]$ instead of $a_i\in[0,p-1]$ is mainly to be consistent with the conventions of \cite{ddr16}. None of the results in this paper depend on this convention, but some combinatorial results may change with a different convention -- it is possible that easier combinatorics was the motivation of this choice in \cite{ddr16} initially.} The $f$-tuple $(a_0,a_1,\dotsc,a_{f-1})$ is uniquely determined by $\restr{\chi}{I_K}$ if we, furthermore, require that $a_i<p$ for at least one $i$. We will call this $f$-tuple the \textbf{tame signature} of $\chi$; this is an element of the set
\[
S:=\{1,2,\dotsc,p\}^f-\{(p,p,\dotsc,p)\}.
\]
We define an action of $\Gal(k/\FFp)=\langle \Frob\rangle\cong \ZZ/f\ZZ$ on the set $S$ via
\[
\Frob\cdot(a_0,a_1,\dotsc,a_{f-1})=(a_{1},a_2,\dotsc,a_{0}).
\]
Note that this notation makes sense since if $\chi$ has tame signature $\vec{a},$ then $\Frob\circ \chi$ has tame signature $\Frob(\vec{a}).$ We define the \textbf{period} of $\vec{a}\in S$ to be the cardinality of its orbit in $S$ under the action of $\Gal(k/\FFp)$. We will call the period of the tame signature of $\chi$ the \textbf{absolute niveau} of the character. We will write $f'$ for the absolute niveau of $\chi$ and $f'':=f/f'$ for the size of the stabiliser of the action of $\Gal(k/\FFp).$ Whenever convenient we will consider the indices of the $a_i$ to be elements of $\ZZ/f\ZZ,$ so $a_f=a_0$ etcetera. We can also write $\restr{\chi}{I_K}=\restr{\omega_i^{n_i}}{I_K}$ for all $0\le i<f$, where we define\footnote{We note that our definition of $n_i$ differs from \cite{ddr16}, since we assumed our embeddings to satisfy $\ovr{\tau}_{i+1}^p=\ovr{\tau}_i$ rather than $\ovr{\tau}_{i}^p=\ovr{\tau}_{i+1}$ as is assumed there.} 
\[
n_i:= \sum_{j=1}^{f} a_{i+j}p^{f-j}=a_{i+1}p^{f-1}+\dots+a_{i+f-1}p+a_i.
\]

\subsubsection{Definition of the filtration}

Suppose $\chi$ is a continuous character $G_K\to\FFpb^\times.$ Recall from \cite[Ch.~IV,\S3]{ser79} that we have a decreasing filtration on $G_K$ by closed subgroups $G_K^u$ given by the upper numbering on $G_K$ for $u\in\RR$. Recall also that $G_K^u=G_K$ for $u\le -1,$ that for $-1< u\le 0$ we have $G_K^u=I_K,$ the inertia subgroup of $G_K,$ and that $\bigcup_{u>0} G_K^u=P_K,$ the wild inertia subgroup of $G_K.$  For $s\in \RR,$ we will define an \textbf{increasing filtration} $\Fil^s$ on $H^1(G_K,\FFpb(\chi))$ by the rule
\[
\Fil^{s}H^1(G_K,\FFpb(\chi)):=\bigcap_{u>s-1} \ker\left(H^1(G_K,\FFpb(\chi))\xrightarrow{\hspace{2mm}\mathrm{res}\hspace{2mm}}H^1(G^u_K,\FFpb(\chi))\right).
\]
We will, furthermore, define
\[
\Fil^{< s}H^1(G_K,\FFpb(\chi)):=\bigcup_{t<s} \Fil^{t} H^1(G_K,\FFpb(\chi))
\]
and define the associated grading by
\[
\gr^s\left(H^1(G_K,\FFpb(\chi))\right):=\frac{\Fil^{s}H^1(G_K,\FFpb(\chi))}{\Fil^{< s}H^1(G_K,\FFpb(\chi))}.
\]

\subsubsection{The filtration for \texorpdfstring{$s \le 1$}{s<=1}}

Let us first study this filtration on the interval $s\in(-\infty,1].$ We see immediately that $\Fil^s H^1(G_K,\FFpb(\chi))=0$ for any $s < 0$. For $0\le s < 1,$ we have
\[
\Fil^{s}H^1(G_K,\FFpb(\chi))=\ker\left(H^1(G_K,\FFpb(\chi))\xrightarrow{\hspace{2mm}\mathrm{res}\hspace{2mm}}H^1(I_K,\FFpb(\chi))\right).
\]
Since $\bigcup_{u>0} G_K^u=P_K,$ a class $c$ of $H^1(G_K,\FFpb(\chi))$ lies in $\Fil^1 H^1(G_K,\FFpb(\chi))$ if and only if $z(P_K)=0$ for any cocycle $z$ representing $c$. Since $P_K$ is the maximal pro-$p$ subgroup of $I_K$, we see that $H^1(I_K/P_K,\FFpb(\chi))=0$ and it follows from inflation-restriction that the restriction map
\[
H^1(I_K,\FFpb(\chi))\to H^1(P_K,\FFpb(\chi))
\] 
is injective. Therefore, any cocycle with trivial restriction to $P_K$ must also have trivial restriction to $I_K$. In other words, 
\[
\Fil^1 H^1(G_K,\FFpb(\chi))=\Fil^0 H^1(G_K,\FFpb(\chi)).
\]

\subsubsection{Filtered piece as kernel}

Suppose that $s>1.$ If $z$ is a cocycle representing a cohomology class in $\Fil^{< s}$, then clearly $z(G_K^{s-1})=0.$ Conversely, if $z$ is a continuous cocycle such that $z(G_K^{s-1})=0$, then it follows from the identity $G_K^{s-1}=\bigcap_{v<s-1} G_K^v$ that 
\[
\ker(z) \cup \bigcup_{v<s-1} \left(G_K-G_K^v\right)
\]
is an open cover of $G_K$. Compactness of $G_K$ then implies that there is a $v_0<s-1$ such that $G_K^{v_0}\subseteq \ker(z).$ In other words, we have just proved that
\[
\Fil^{< s}H^1(G_K,\FFpb(\chi))=\ker\left(H^1(G_K,\FFpb(\chi))\xrightarrow{\hspace{2mm}\mathrm{res}\hspace{2mm}}H^1(G^{s-1}_K,\FFpb(\chi))\right).
\]

\subsubsection{The jumps in the filtration}\label{subsec:filjmp} Having defined the filtration and its most basic properties, we are now interested in the jumps of the filtration and the size of these jumps. That is, we would like to study the dimensions of the spaces $\gr^s\left(H^1(G_K,\FFpb(\chi))\right)$. Fix a continuous character $\chi:G_K\to\FFpb^\times$ and let the $n_i$ for $i=0,\dotsc,f-1$ be as above. The following theorem gives the answer.

\begin{thm}\label{thm:ddr3.1}
For $s\in\RR$, write $d_s:=\dim_{\FFpb}\gr^s\left(H^1(G_K,\FFpb(\chi))\right).$ Then $d_s=0$ unless $s=0$ or $1< s \le 1+\frac{pe}{p-1}.$ Moreover, if $1< s < 1+\frac{pe}{p-1}$ and $d_s\neq 0$ then $s=1+\frac{m}{p^f-1}$ for an integer $m\not\equiv 0 \bmod p$. More precisely, the dimensions $d_s$ are given by
\begin{enumerate}[\hspace{.5cm}(1)]
\item \label{pt1} $d_0=1$ if $\chi$ is trivial and $d_0=0$ otherwise.

\item \label{pt2} if $1<s<1+\frac{pe}{p-1}$ and $s=1+\frac{m}{p^f-1}$ for $p\nmid m$, then
\[
d_s=\#\left\{i\in\{0,\dotsc,f-1\}\mid m\equiv n_i \bmod{p^f-1}\right\}.
\]

\item \label{pt3} $d_{1+\frac{ep}{p-1}}=1$ if $\chi$ is cyclotomic and $d_{1+\frac{ep}{p-1}}=0$ otherwise. 
\end{enumerate}
\end{thm}

\begin{rem}\label{rem:dim-is-f/f'}
Note that if $t,u\in\{0,\dotsc,f-1\}$ are such that $n_t$ and $n_u$ are both congruent to $m\bmod{(p^f-1)}$ then the uniqueness of the $n_i$ modulo $(p^f-1)$ implies that $n_t=n_u$. Hence, either $t=u$ or $\vec{a}$ has a non-trivial stabilizer in $\Gal(k/\FFp).$ This observation implies that, for $1<s<1+\frac{pe}{p-1},$ we have that $d_s=f/f'$ whenever $d_s>0$, where $f'$ is the absolute niveau of $\chi$.
\end{rem}

For the theorem to make sense, of course, the sum of the dimensions $d_s$ must add up to the $\FFpb$-dimension of $H^1(G_K,\FFpb(\chi)).$ Since we will need this in the proof of the theorem, let us prove it first. To contrast the actual dimensions of the graded pieces with the dimensions claimed in the second part of the theorem above, let us write $d_s'$ for the values of the dimensions claimed in the theorem.

\begin{prop}\label{prop:dimcount} For $j=0,\dotsc,e-1,$ we have
\[
\sum_{\frac{jp}{p-1}<\frac{m}{p^f-1}<\frac{(j+1)p}{p-1}} d_{1+\frac{m}{p^f-1}}'=f.
\]
\end{prop}
\begin{rem}
In fact, the proposition and Remark \ref{rem:dim-is-f/f'} above imply that for any $j\ge 0$ the set of integers $m\in\ZZ$ satisfying
\begin{enumerate}
\item $\frac{jp}{p-1}<\frac{m}{p^f-1}<\frac{(j+1)p}{p-1}$,
\item $p\nmid m$ and 
\item there exists an $i$ such that $m\equiv n_i\bmod{(p^f-1)}$
\end{enumerate}
has cardinality $f',$ where $f'$ is the absolute niveau of $\chi.$ We will need this later in \S\ref{sec:expl-basis}.
\end{rem}
\begin{proof}
We have the inequalities $\frac{1}{p-1}\le \frac{n_i}{p^f-1} < \frac{p}{p-1}$ for all $i.$ It is easy to show explicitly that there exists a unique integer $k_j$ such that
\[
\left(\left[\frac{1}{p-1},\frac{p}{p-1}\right)+k_j\right) \cap \left(\frac{jp}{p-1},\frac{(j+1)p}{p-1}\right) \neq \varnothing
\]
and
\[
\left(\left[\frac{1}{p-1},\frac{p}{p-1}\right)+k_j+1\right) \cap \left(\frac{jp}{p-1},\frac{(j+1)p}{p-1}\right) \neq \varnothing.
\]
Moreover, if we write $j-1\equiv b \bmod{p-1}$ for $0\le b <p-1$ (when $p=2$ we simply let $b=0$), then we can show $k_j\equiv b+1 \bmod p.$ We are interested in counting the number of $n_i$ that satisfy
\begin{equation}\label{eqn:1}
\frac{jp}{p-1} < \frac{n_i}{p^f-1}+k_j\text{ and }n_i+k_j(p^f-1)\not\equiv 0 \bmod{p}
\end{equation}
or
\begin{equation}\label{eqn:2}
\frac{n_i}{p^f-1}+(k_j+1) < \frac{(j+1)p}{p-1}\text{ and }n_i+(k_j+1)(p^f-1)\not\equiv 0 \bmod{p}
\end{equation}
or both, in which case they should be counted twice. We don't have to worry about upper or lower bounds in the first or second case, respectively, since these are automatically satisfied.

In the case of Equation (\ref{eqn:1}) the inequality reduces to the inequality
\[
\frac{n_i}{p^f-1}>\frac{b+1}{p-1}.
\]
This inequality is satisfied if and only if $n_i$ has coordinates 
\[
(a_i,a_{i+1},\dotsc,a_j,\dotsc)=(a_i, b+1, b+1,\dotsc,b+1,a_j,\dotsc)
\]
with $a_j>b+1.$ If we, furthermore, require that $n_i+k_j(p^f-1)\not\equiv 0 \bmod{p},$ we see that this requirement is equivalent to requiring that $a_i\neq b+1.$ Hence, the number of such $n_i$ is the same as the number $\#\{a_j>b+1\}.$

In the case of Equation (\ref{eqn:2}) the inequality reduces to the inequality
\[
\frac{n_i}{p^f-1}<\frac{b+2}{p-1}.
\]
This inequality is satisfied if and only if $n_i$ corresponds to a tuple 
\[
(a_i,a_{i+1},\dotsc,a_j,\dotsc)=(a_i,b+2,b+2,\dotsc,b+2, a_j,\dotsc)
\]
with $a_j\le b+1$. If we also require $n_i+(k_j+1)(p^f-1)\not\equiv 0 \bmod{p},$ we find that $a_i\neq b+2.$ Hence, in this case we see that the number of $n_i$ satisfying both conditions is $\#\{a_j \le b+1\}.$

It follows that the total number of $n_i$ satisfying one of these equations, counted with multiplicity, is 
\[
\#\{a_j>b+1\}+\#\{a_j \le b+1\}=f.
\]
\end{proof}

A straightforward calculation using the local Euler--Poincar\'e characteristic gives 
\[
\dim_{\FFpb} H^1(G_K,\FFpb(\chi))=
\begin{cases} 
ef+2 &\text{if $\chi$ is both trivial and cyclotomic;} \\
ef+1 &\text{if $\chi$ is trivial, but not cyclotomic;} \\
ef+1 &\text{if $\chi$ is cyclotomic, but not trivial;} \\ 
ef &\text{otherwise.} \end{cases}
\]
From the proposition above and this observation we immediately get the following corollary.
\begin{cor}\label{cor:dimcount}
Let the $d_s'$ denote the dimensions as claimed in Theorem \ref{thm:ddr3.1}. Then
\[
\sum_{1<s<1+\frac{ep}{p-1}} d_s'=ef.
\]
Hence, we find $\sum_s d_s' =\dim_{\FFpb} H^1(G_K,\FFpb(\chi)).$
\end{cor}

\begin{proof}[Proof of Theorem \ref{thm:ddr3.1}]
Let us write $d_s'$ again for the values of $d_s$ claimed in the second part of Theorem \ref{thm:ddr3.1}. By Corollary \ref{cor:dimcount} we know that $\sum_s d_s'=\sum_s d_s$, hence it suffices to consider only $s\in\RR$ such that $d_s'>0$ and prove $d_s'=d_s$ for these values of $s.$

Let us first treat the case $s=0.$ We already saw above that
\[
\Fil^{0}H^1(G_K,\FFpb(\chi))=\ker\left(H^1(G_K,\FFpb(\chi))\xrightarrow{\hspace{2mm}\mathrm{res}\hspace{2mm}}H^1(I_K,\FFpb(\chi))\right)
\]
and $\Fil^s=0$ for $s<0.$ We can use the inflation-restriction exact sequence to show that the kernel above is isomorphic to $H^1\left(G_K/I_K,\FFpb(\chi)^{I_K}\right).$ This space is clearly 1-dimensional over $\FFpb$ if $\chi$ is trivial and $0$-dimensional if $\chi$ is ramified. However, if $\chi$ is unramified and non-trivial, by identifying $G_K/I_K$ with $\hat{\ZZ}$, we may consider $\chi$ as a character $\hat{\ZZ}\to \FFpb^\times$. By continuity any representation of $\hat{\ZZ}$ is completely determined by the image of $1.$ An extension of $\FFpb$ by $\FFpb(\chi)$ is always split since the matrix corresponding to 1 has distinct eigenvalues over an algebraically closed field (as $\chi$ is non-trivial). Hence, $\dim_{\FFpb} H^1\left(G_K/I_K,\FFpb(\chi)^{I_K}\right)=1$ if $\chi$ is trivial and $0$ otherwise.

Assume that $1<s\le 1+\frac{ep}{p-1}$ and $m:=(s-1)(p^f-1).$ We define an extension $M/K$ by letting $M:=L(\pi)$ for $\pi$ a root of $x^{p^f-1}+\pi_K$ and $L/K$ an unramified extension of degree prime to $p$ such that $\restr{\chi}{G_M}$ is trivial. The extension $M$ is now a tamely ramified Galois extension of $K$ of ramification degree $p^f-1$ and $\Gal(M/K)$ has order prime to $p$. Combined with the inflation-restriction sequence
\[
0\to H^1(\Gal(M/K),\FFpb(\chi))\to H^1(G_K,\FFpb(\chi))\to H^1(G_M,\FFpb(\chi))^{\Gal(M/K)}
\]
this gives
\[
H^1(G_K,\FFpb(\chi))\cong H^1(G_M,\FFpb(\chi))^{\Gal(M/K)}=\Hom_{\Gal(M/K)}(G_M^{\mathrm{ab}},\FFpb(\chi)).
\]
The isomorphisms of local class field theory give an isomorphism
\[
G_M^{\mathrm{ab}}\otimes \FFp \cong M^\times/(M^\times)^p,
\]
and we obtain
\[
H^1(G_K,\FFpb(\chi))\cong \Hom_{\Gal(M/K)}(M^\times/(M^\times)^p,\FFpb(\chi)).
\]
Since $M/K$ is tamely ramified of ramification degree $p^f-1$, we know that Hasse-Herbrand function is given by $\psi_{M/K}(x)=(p^f-1)x$ for $x\ge 0$ and that $G_K^u\subset G_M$ for any $u>0.$ Using Proposition IV.15 \cite[p.~74]{ser79}, it follows that for any finite extension $M'/M$ and any $u>0$ we have 
\[
\Gal(M'/K)^u=\Gal(M'/M)_{\psi_{M'/M}((p^f-1)u)}=\Gal(M'/M)^{(p^f-1)u}.
\]
So $G_K^u=G_M^{(p^f-1)u}$ for $u>0.$ Moreover, the maps of local class field theory map $G_M^{(p^f-1)u}$ onto the unit group $1+\pi^{\ceil{(p^f-1)u}}\O_M$ by Corollary 3 to Theorem XV.1 in \cite[p.~228]{ser79}. If we denote the unit group $1+\pi^{i}\O_M$ by $U_i$ for any integer $i>0$, then it becomes apparent that a cocycle in $c\in H^1(G_K,\FFpb(\chi))$ has trivial restriction to $G_K^u$ for all $u>s-1$ (to $G_K^{s-1}$, resp.) if and only if the corresponding homomorphism $M^\times/(M^\times)^p\to\FFpb(\chi)$ factors through $M^\times/(M^\times)^pU_{m+1}$ (through $M^\times/(M^\times)^pU_{m}$, resp.), where we still let $m=(s-1)(p^f-1).$ Note that this is precisely the condition such that $c\in\Fil^s H^1(G_K,\FFpb(\chi))$ ($c\in\Fil^{<s} H^1(G_K,\FFpb(\chi)),$ resp.) and it follows that
\[
\gr^s\left(H^1(G_K,\FFpb(\chi))\right)\cong \Hom_{\Gal(M/K)}\left(\frac{U_m}{(U_m\cap (M^\times)^p) U_{m+1}},\FFpb(\chi)\right).
\]

First assume $1<s<1+\frac{ep}{p-1}$ and $p\nmid m:=(s-1)(p^f-1).$ This is equivalent to the requirement that $0<m<\frac{ep(p^f-1)}{(p-1)}.$ We claim that $U_m\cap (M^\times)^p\subset U_{m+1}.$ Let $x=1+u\pi^t$ for positive integer $t$ and a unit $u\in\O_M^\times$ and suppose $v_M(x^p-1)\ge m.$ But
\[
v_M(x^p-1)=v_M(pu\pi^t+\dots+u^p\pi^{pt})\ge \min(e(p^f-1)+t,pt)
\]
with equality unless $e(p^f-1)+t=pt.$ If $pt<e(p^f-1)+t$, then $m=pt$, which would contradict the condition $p\nmid m.$ On the other hand, if $e(p^f-1)+t\le pt$, then $t\ge \frac{e(p^f-1)}{p-1}$ gives $e(p^f-1)+t\ge \frac{ep(p^f-1)}{(p-1)}$ and, therefore, $m=e(p^f-1)+t$ would contradict the upper bound on $m.$ Hence, $v_M(x^p-1)\ge m+1$ as required. It follows from the claim that
\[
\gr^s\left(H^1(G_K,\FFpb(\chi))\right)\cong \Hom_{\Gal(M/K)}\left(U_m/U_{m+1},\FFpb(\chi)\right).
\]
We let $l$ denote the residue field of $L$. The action of $\sigma\in\Gal(M/K)$ on $U_m/U_{m+1}$ sends $1+x\pi^m\mapsto 1+\sigma(x)\overline{\omega}_\pi(\sigma)\pi^m.$ Therefore, we have a $\Gal(M/K)$-equivariant isomorphism
\[
\begin{tikzcd}[column sep=2.5cm]
l(\overline{\omega}_\pi^m) \ar[r, "x\xmapsto{\hspace{8mm}} 1+x\pi^m"'] & U_m/U_{m+1}.
\end{tikzcd}
\]
Let $S_i$ be the set of embeddings $l\to\FFpb$ which restrict to $\overline{\tau}_i$ over $k$. Then the map
\begin{align*}
l(\overline{\omega}_\pi^m)\otimes_{\FFp} \FFpb &\cong \bigoplus_{i=0}^{f-1}\left( \bigoplus_{\tau\in S_i} \FFpb(\omega_i^m)\right) \\
x\otimes 1&\mapsto (\tau(x))_\tau
\end{align*}
is a $\Gal(M/K)$-equivariant isomorphism, where the action of $\Gal(M/K)$ on $\bigoplus_{\tau\in S_i}\FFpb(\omega_i^m)$ is defined by $g\cdot ((x_\tau)_\tau)=(\omega_i^m(g)x_{\tau\circ g})_{\tau}.$ (This is the well-known isomorphism 
\begin{align*}
l\otimes_{\FFp}\FFpb&\cong \bigoplus_{\tau:l\hookrightarrow \FFpb} \FFpb ;\\
x\otimes 1 &\mapsto (\tau(x))_\tau,
\end{align*}
in which we keep track of the Galois action.) We note that 
\[
\bigoplus_{\tau\in S_i} \FFpb \cong \mathrm{Ind}_{\Gal(M/L)}^{\Gal(M/K)} \FFpb.
\]
The latter representation is just the regular representation of the quotient $\Gal(L/K)$. Since this group is finite abelian of order prime-to-$p$, its regular representation decomposes as a direct sum of all its 1-dimensional representations, i.e.
\[
\mathrm{Ind}_{\Gal(M/L)}^{\Gal(M/K)} \FFpb \cong \bigoplus_{\nu:\Gal(L/K)\to\FFpb^\times} \FFpb(\nu),
\] 
and hence we find that
\[
l(\overline{\omega}_\pi^m)\otimes_{\FFp} \FFpb \cong \bigoplus_{i=0}^{f-1}\left(\bigoplus_{\nu} \FFpb(\nu\omega_i^m) \right).
\]
It follows that the dimension of $\gr^s\left(H^1(G_K,\FFpb(\chi))\right)$ is equal to 
\[
\dim_{\FFpb} \Hom_{\Gal(M/K)}\left(\bigoplus_{i=0}^{f-1}\left(\bigoplus_{\nu} \FFpb(\nu\omega_i^m) \right),\FFpb(\chi)\right)
\]
and this dimension is equal to the number of $i$ such that $m\equiv n_i \bmod{p^f-1},$ which proves that $d_s'=d_s$ for all $1<s<1+\frac{ep}{p-1}$ for which $d_s'>0.$

Finally, assume $s=1+\frac{ep}{p-1},$ or, equivalently, assume $m=\frac{ep(p^f-1)}{p-1}.$ We may, furthermore, assume that $\chi$ is cyclotomic, since we are only considering cases in which $d_s'>0.$ Note that the requirement that $\restr{\chi}{G_M}$ is trivial implies that $\zeta_p\in M.$ To show $d_s=d_s'= 1$, by the earlier calculation, is equivalent to showing
\[
\Hom_{\Gal(M/K)}\left(\frac{U_m}{(U_m\cap (M^\times)^p) U_{m+1}},\FFpb(\chi)\right)
\]
is 1-dimensional. We will prove that the quotient of unit groups is isomorphic to $\FFp(\chi)$ as a module over $\FFp[\Gal(M/K)]$ from which the required result is immediate. Similarly to the calculation on valuations above, it is easy to show $U_m\cap (M^\times)^p=(U_{n})^p$ where $n=m/p.$ We can obtain the quotient $\frac{U_m}{(U_m\cap (M^\times)^p) U_{m+1}},$ therefore, as the cokernel of the $p$-power map
\[
\begin{tikzcd}[column sep=3cm]
U_n/U_{n+1} \arrow[r, "x\xmapsto{\hspace{1cm}}x^p"'] & U_{m}/U_{m+1}.
\end{tikzcd}
\]
A small calculation shows that this map sends $1+s\pi^n\mapsto 1+(s^p+cs)\pi^m,$ where $c\in\O_M^\times$ is defined by $p=c\pi^{e(p^f-1)}.$ Since $\zeta_p\in M$, we know that $M$ contains $N:=\QQp(\zeta_p)=\QQp(\sqrt[p-1]{-p}).$ Therefore, there exists a unit $u\in \O_M^\times$ such that $\pi^{e(p^f-1)}=(u\sqrt[p-1]{-p})^{p-1}=-u^{p-1}p$ giving $-c=(u^{-1})^{p-1}$. We conclude that $-c$ always has a $(p-1)$st root in the residue field and that the $p$-power map above has kernel and, thus, cokernel of size $p.$ Write $\overline{s}$ for the image of any $s\in \O_M^\times$ in the cokernel. The action of $\Gal(M/K)$ on this cokernel is given by $\overline{\omega}_\pi^m\otimes \mu_{-\overline{c}}$ where $\mu_\alpha:\Gal(M/K)\to \FFp^\times$ is the unramified character sending the (absolute) arithmetic Frobenius $\Frob_p$ to $\alpha\in \FFp^\times.$ For any uniformiser $\pi_N$ of $N,$ we can write $\zeta_p=v\pi_N+1$ for a unit $v\in\O_N^\times.$ Hence, one sees that for $\sigma\in \Gal(M/K)$ we have 
\[
\frac{\sigma(v\pi_N)}{v\pi_N}=\frac{(1+v\pi_N)^{\chi(\sigma)}-1}{v\pi_N}\equiv \chi(\sigma) \bmod{\pi}.
\]
Since $\sigma(v)/v\equiv 1\bmod{\pi},$ it follows that $\sigma(\pi_N)/\pi_N \equiv \chi(\sigma)\bmod \pi$ for any uniformiser $\pi_N$ of $N$, in particular $\pi_N=\sqrt[p-1]{-p}.$ Thence, if $\sigma\in\Gal(M/K)$ is a lift of $\Frob_p,$ then
\[
\overline{\omega}_\pi(\sigma)^n=\frac{\sigma(u\sqrt[p-1]{-p})}{u\sqrt[p-1]{-p}} \equiv \overline{u}^{p-1}\chi(\sigma)=(-\overline{c})^{-1}\chi(\sigma).
\]
In other words, we've just shown that $\overline{\omega}_\pi^n$ acts as $\chi\otimes\mu_{-\overline{c}^{-1}}$ on the cokernel and, therefore, $\overline{\omega}_\pi^{m}$ must act on the cokernel in the same way since $(p^f-1)\mid m-n.$ Thus, $\Gal(M/K)$ acts on the cokernel via the cyclotomic character and $U_m/(U_m\cap (M^\times)^p) U_{m+1}\cong \FFp(\chi),$ as required.
\end{proof}

\subsection{Constructing the basis elements}\label{sec:bas-constr}

In this subsection we will construct explicit basis elements of the space $H^1(G_K,\FFpb(\chi)),$ which we will later use to give an explicit version of Serre's conjecture. We will construct the basis by making certain choices and only at the end of the next section it will follow that our construction is independent of the choices made.

\subsubsection{The Artin-Hasse Exponential}

Before we move on to define a basis of $H^1(G_K,\FFpb(\chi))$ explicitly, we must first define a homomorphism $\varepsilon_\alpha$ in terms of which our basis will be defined. Recall that the \textbf{Artin-Hasse exponential} is given by
\[
E_p(x)=\exp\left(\sum_{n\ge 0} \frac{x^{p^n}}{p^n}\right).
\]
It's a well-known fact that $E_p(x)$ is an element of $\ZZp\llbracket x\rrbracket$ (see, for example, \cite[\S7.2.2]{rob00}), i.e. its coefficients are $p$-integral elements of $\QQ.$ Since $p$ is fixed throughout, we will write $E(x)$ instead of $E_p(x)$ from now on.

Suppose $l$ is a finite field of characteristic $p$ and $L$ is the field of fractions of the ring of Witt vectors $W(l)$ of $l$. Let $M$ be a subfield of $\mathbf{C}_p$ containing $L$ and let $\alpha\in M$ be such that $|\alpha|<1.$ We define a map
\begin{align}\label{eqn:ahexphom}
\begin{split}
\varepsilon_\alpha:l\otimes_{\FFp}\FFpb &\to \O_M^\times \otimes_{\ZZ} \FFpb \\ a\otimes b &\mapsto E([a]\alpha)\otimes b,
\end{split}
\end{align}
where $[a]$ is a Teichm\"uller lift of $a\in l.$ It follows from \cite[Lem.~4.1]{ddr16} that this map is a homomorphism which relates the additive structure of $l$ to the multiplicative structure of $\O_M^\times\otimes \FFp.$ We will use this homomorphism to construct an explicit basis for $H^1(G_K,\FFpb(\chi)).$

\subsubsection{An \texorpdfstring{$\FFpb$}{Fp-bar}-dual of \texorpdfstring{$H^1(G_K,\FFpb(\chi))$}{H1}}\label{subsec:fpb-dual}

We need a slightly more general set-up than in the previous subsection. Suppose $K/\QQp$ is a finite extension of ramification index $e,$ of residue degree $f$ and with residue field $k$. Moreover, we take $\chi:G_K\to\FFpb^\times$ to be any continuous character. Let $M=L(\pi)$ be a totally tamely ramified extension of an unramified prime-to-$p$-degree extension $L/K$, where the ramification degree $e_M$ of $M/K$ satisfies $e_M\mid p^f-1$ and the uniformiser $\pi$ of $M$ satisfies $\pi^{e_M}\in K^\times,$ and we take $M$ sufficiently large so that $\restr{\chi}{G_M}$ is trivial. Note that $e_M$ denotes the ramification degree of $M$ over $K$, and the ramification degree $e_{M/\QQp}$ of $M$ over $\QQp$ is given by $e_{M/\QQp}=e_M e.$ We recover the set-up of the previous subsection by taking $e_M=p^f-1$ and $\pi^{e_M}=-\pi_K.$ Note that the proof of Theorem \ref{thm:ddr3.1} follows in the same way as before, taking into account that now $s=1+\frac{m}{e_M}.$ In particular, for $1<s\le 1+\frac{ep}{p-1}$ we still have the isomorphism
\[
\gr^s\left(H^1(G_K,\FFpb(\chi))\right)\cong \Hom_{\Gal(M/K)}\left(\frac{U_m}{(U_m\cap (M^\times)^p) U_{m+1}},\FFpb(\chi)\right),
\]
which, if we furthermore require that $s<1+\frac{ep}{p-1}$ and $p\nmid m:=e_M(s-1)$, can again be shown to simplify to
\[
\gr^s\left(H^1(G_K,\FFpb(\chi))\right)\cong \Hom_{\Gal(M/K)}\left(U_m / U_{m+1},\FFpb(\chi)\right).
\]
More generally, for this choice of $M$ we certainly still have
\begin{align*}
H^1(G_K,\FFpb(\chi))&\cong \Hom_{\Gal(M/K)}\left(M^\times,\FFpb(\chi)\right) \\
&\cong \Hom_{\FFpb}\left(\left(M^\times\otimes_{\ZZ}\FFpb(\chi^{-1})\right)^{\Gal(M/K)},\FFpb\right).
\end{align*}
We will construct our basis of $H^1(G_K,\FFpb(\chi))$ as a dual basis to a basis for the space above, therefore it will be useful to give this space a name. Let
\[
U_\chi:=\left(M^\times\otimes_{\ZZ}\FFpb(\chi^{-1})\right)^{\Gal(M/K)}.
\]

\subsubsection{The basis elements \texorpdfstring{$u_{\alpha}$}{u-alpha}}\label{subsec:u-ij}

Letting $f'$ denote the absolute niveau of $\chi$ and $f'':=f/f',$ recall from Section~\ref{subsec:filjmp} that for any $0\le j<e$ the set of integers $m\in\ZZ$ satisfying
\begin{enumerate}
\item $\frac{jp}{p-1}<\frac{m}{p^f-1}<\frac{(j+1)p}{p-1}$,
\item $p\nmid m$ and 
\item there exists an $i$ such that $m\equiv n_i\bmod{(p^f-1)}$
\end{enumerate}
has cardinality $f'.$ Denote this set of integers by $W_j'$ and let $W':=\bigcup_{j=0}^{e-1} W_j',$ which is now a set of integers of cardinality $ef'.$ For any $m\in W'$, we get a uniquely defined $i_{m}\in\{0,\dotsc,f'-1\}$ via the association $m\equiv n_{i_{m}}\bmod{(p^f-1)}.$ (Note that there may be distinct $m_1\neq m_2\in W'$, possibly even lying in the same component $W_j'$, that satisfy $i_{m_1}=i_{m_2}.$) We define $W:=W'\times \{0,\dotsc,f''-1\}$ which has cardinality $ef$. To any $\alpha=(m,k)\in W$ (with $m\in W'$ and $0\le k <f''$) we can attach an embedding via $\tau_\alpha:=\tau_{i_m+kf'}.$ Notice that we also have the congruence $m\equiv n_{i_m+kf'}\bmod{(p^f-1)}$ for any $k\in \ZZ.$

Let us return to the situation where $M=L(\pi)$ for an unramified extension $L$ of $K$ of degree prime-to-$p$ and a uniformiser $\pi$ such that $\pi^{e_M}\in K^\times$ and $\restr{\chi}{G_M}$ is trivial. This implies that $\restr{\chi}{I_K}$ factors through $I_{M/K}$ which has cardinality $e_M$, so for any $i$ we see $(\restr{\chi}{I_K})^{e_M}=\omega_i^{n_ie_M}=\mathbf{1}.$ Any $n_i$, therefore, is divisible by $(p^f-1)/e_M.$ It follows that the set of integers $m'\in\ZZ$ satisfying the three conditions
\begin{enumerate}
\item $\frac{jp}{p-1}<\frac{m'}{e_M}<\frac{(j+1)p}{p-1}$,
\item $p\nmid m'$ and 
\item there exists an $i$ such that $m'\equiv \frac{e_M n_i}{p^f-1}\bmod{e_M}$
\end{enumerate}
is in bijection with the set $W_j'$ above by the explicit map $m\mapsto m':=\frac{e_M m}{p^f-1}.$ Unfortunately, this notation leads to the awkward situation that, in general, $m'\notin W',$ but we will stick to it because $m'$ is analogous to the integers $n_i'$ appearing in \cite{ddr16}. We will write $\omega_\pi$ for the character $G_K\to K^\times$ defined by $\sigma\mapsto \sigma(\pi)/\pi$ and $\overline{\omega}_\pi$ for its reduction mod $\pi_K$. We note that this character factors through $\Gal(M/K)$ and, since $\sigma(\pi)$ is a root of $x^{e_M}-\pi^{e_M}\in K[x],$ the image of $\omega_\pi$ is contained in $\mu_{e_M}(K),$ the $e_M$th roots of unity in $K^\times.$ To contrast this with our previous fundamental character, let us from now on write $\ovr{\omega}_{\mathrm{fc}}:G_K\to k^\times$ for the mod $p$ character defined by $\sigma\mapsto \sigma(\beta)/\beta$ for $\beta$ a root of $x^{p^f-1}+\pi_K.$ We have $\restr{\overline{\omega}_\pi}{I_K}=(\restr{\ovr{\omega}_{\mathrm{fc}}}{I_K})^{(p^f-1)/e_M},$ and, therefore, for all $0 \le i <f,$ we have
\[
\restr{\chi}{I_K}=(\tau_i\circ \overline{\omega}_\pi)|_{I_K}^{n_ie_M/(p^f-1)}.
\]
Writing $m':=m e_M/(p^f-1),$ we see that for any $\alpha=(m,k)\in W$ we have
\[
\chi=\mu(\tau_{\alpha}\circ \overline{\omega}_\pi)^{m'}
\]
for a unique unramified character $\mu:\Gal(L/K)\to \FFpb^\times$ independent of $\alpha.$

Recall that we have an isomorphism
\[
l\otimes_{\FFp}\FFpb \cong \bigoplus_{\tau :k\hookrightarrow\FFpb} l\otimes_{k,\tau}\FFpb.
\]
By the normal basis theorem $l$ is free of rank 1 over $k[\Gal(l/k)]$ and, thus, $l\otimes_{k,\tau}\FFpb$ is free of rank 1 over $\FFpb[\Gal(l/k)]$. The latter statement means that $l\otimes_{k,\tau}\FFpb$ is isomorphic to the regular representation of $\Gal(l/k)$ over $\FFpb$, which splits up in a direct sum over all characters by Maschke's Theorem. Hence, for each embedding $\tau,$ the $\mu$-eigenspace $\Lambda_{\tau,\mu}$ defined as
\[
\left\{a\otimes b\in l\otimes_{k,\tau}\FFpb\;\middle|\; g(a)\otimes b =(1\otimes\mu(g))a\otimes b\text{ for all }g\in\Gal(L/K)\right\}
\]
is 1-dimensional over $\FFpb.$ We let $\lambda_{\tau,\mu}$ be any non-zero element of $\Lambda_{\tau,\mu}$; note that these elements are linearly independent over $\FFpb$ for differing embeddings $\tau.$

Then, for any $\alpha=(m,k)\in W$, we define
\[
u_{\alpha}:=\varepsilon_{\pi^{m'}}(\lambda_{\tau_{\alpha},\mu})\in \O_M^\times\otimes_{\ZZ} \FFpb,
\]
where $\varepsilon_{\pi^{m'}}$ is defined by (\ref{eqn:ahexphom}), giving $ef$ elements $u_\alpha.$ Later we will prove that the elements $u_\alpha$ give the explicit basis of $U_\chi$ that we are pursuing. In order for this even to make sense, we will now show $u_\alpha\in U_\chi.$ For any $g\in\Gal(M/K),$ it follows from the observation $\omega_\pi(g)=[\overline{\omega}_\pi(g)]$ that 
\[
g\cdot E([a]\pi^{m'})=E\left([g(a)\overline{\omega}_\pi(g)^{m'}]\pi^{m'}\right)
\]
for any $a\in l.$ Thus, for any $\alpha=(m,k)\in W$,
\begin{align*}
g\cdot\varepsilon_{\pi^{m'}}(\lambda_{\tau_{\alpha},\mu})&=\varepsilon_{\pi^{m'}}((\overline{\omega}_\pi(g)^{m'}\otimes 1)g(\lambda_{\tau_{\alpha},\mu})) \\
&=\varepsilon_{\pi^{m'}}((\overline{\omega}_\pi(g)^{m'}\otimes \mu(g))\lambda_{\tau_{\alpha},\mu}).
\end{align*}
Since $\overline{\omega}_\pi(g)^{m'}\otimes \mu(g)=1\otimes\mu(g)(\tau_{\alpha}\circ\overline{\omega}_\pi(g))^{m'}=1\otimes \chi(g)$ in $l\otimes_{k,\tau_{\alpha}}\FFpb$, we see that $g u_{\alpha}=(1\otimes\chi(g))u_{\alpha}$ for all $g\in\Gal(M/K).$ Therefore, we can view $u_{\alpha}$ as an element of $U_\chi.$

\subsubsection{The definitions of \texorpdfstring{$u_{\mathrm{triv}}$}{u-triv} and \texorpdfstring{$u_{\mathrm{cyc}}$}{u-cyc}}

The definition of $u_{\mathrm{triv}}$ is straightforward. For $g\in\Gal(M/K),$ it is clear that $g(\pi)=\omega_\pi(g)\pi.$ As noted above, the image of $\omega_\pi$ is contained in $\mu_{e_M}(K).$ Since $p$ is coprime to $e_M$, the $p$-th power map gives an isomorphism $\mu_{e_M}(K)\xrightarrow{\hspace{2mm}\sim\hspace{2mm}}\mu_{e_M}(K).$ Thus, in fact, $\omega_\pi(g)\in (M^\times)^p$ and we set
\[
u_{\mathrm{triv}}:=\pi\otimes 1\in M^\times \otimes_{\FFp}\FFpb,
\]
which is an element of $U_\chi$ when $\chi$ is the trivial character.

Suppose $\chi$ is cyclotomic. Recall from the last part of the proof of Theorem \ref{thm:ddr3.1} that we have already proved explicitly that, for $m=\frac{epe_M}{p-1},$ we have
\[
\frac{U_m}{(U_m\cap (M^\times)^p)U_{m+1}}\cong \FFp(\chi).
\]
However, recall that the $p$-adic exponential converges on values of $y\in M$ such that $v_M(y)>e_Me/(p-1).$ For $1+x\in U_{m+1},$ we have $v_M(\log(1+x))> m$ and, thus, we see that $\exp(\log(1+x)/p)$ converges to a $p$-th root of $1+x.$ Therefore, $U_{m+1}\subset (M^\times)^p$ and also $U_{m+1}\subset (U_m\cap (M^\times)^p).$ Thus, we have an injection of the quotient above into $\O_M^\times \otimes_{\ZZ} \FFp.$

We extend the injection by scalars to $\FFpb$ and define $u_{\mathrm{cyc}}$ to be any non-trivial element of
\[
u_{\mathrm{cyc}}\in \mathrm{Im}\left(U_m \otimes_{\ZZ} \FFpb \xrightarrow{\hspace{5mm}} \O_M^\times \otimes_{\ZZ} \FFpb\right).
\]
It is now obvious from the isomorphism with $\FFp(\chi)$ that for any element $g\in\Gal(M/K)$ we have that
\[
g(u_{\mathrm{cyc}})=(\chi(g)\otimes 1)u_{\mathrm{cyc}}=(1\otimes \chi(g))u_{\mathrm{cyc}},
\]
hence, indeed, $u_{\mathrm{cyc}}\in U_\chi.$

\subsection{An explicit basis of \texorpdfstring{$H^1(G_K,\FFpb(\chi))$}{H1}}\label{sec:expl-basis}

In this section we will prove that the elements as constructed in the previous section give an explicit basis of $H^1(G_K,\FFpb(\chi))$. We will do this by proving the following theorem.

\begin{thm}
Let $B$ denote the elements $u_{\alpha}$ for $\alpha \in W$ with, additionally, the elements $u_{\mathrm{triv}}$ if $\chi$ is trivial and $u_{\mathrm{cyc}}$ if $\chi$ is cyclotomic. Then $B$ is an $\FFpb$-basis of $U_\chi.$
\end{thm}
\begin{proof} To prove this theorem, we will define a decreasing filtration on $U_\chi,$ which is dual to the increasing filtration on $H^1(G_K,\FFpb(\chi))$ defined earlier. Let $\Fil^0U_\chi=U_\chi$ and let $\Fil^m U_\chi$ denote the image of the map
\[
\begin{tikzcd}
(U_m\otimes \FFpb(\chi^{-1}))^{\Gal(M/K)} \arrow{r} & (M^\times \otimes \FFpb(\chi^{-1}))^{\Gal(M/K)}=:U_\chi
\end{tikzcd}
\]
for $m\ge 1.$ Let $\gr^mU_\chi:=\Fil^mU_\chi/\Fil^{m+1} U_\chi$ as usual. It is immediate by duality that $\dim_{\FFpb}U_\chi=\dim_{\FFpb}H^1(G_K,\FFpb(\chi)).$ More specifically, whenever $0<m<\frac{pee_M}{p-1}$ and $p\nmid m$, it follows as in the proof of Theorem \ref{thm:ddr3.1} and the discussion in \S\ref{subsec:fpb-dual} that
\[
\gr^s\left(H^1(G_K,\FFpb(\chi))\right)\cong \Hom_{\Gal(M/K)}\left(U_m/U_{m+1},\FFpb(\chi)\right),
\]
where $s:=1+\frac{m}{e_M}.$ We claim that 
\[
\gr^m U_\chi \cong \left(U_m/U_{m+1} \otimes_{\FFp} \FFpb(\chi^{-1})\right)^{\Gal(M/K)}.
\]
In the proof of Theorem \ref{thm:ddr3.1} we show that under these assumptions on $m$ we have that $(M^\times)^p \cap U_m \subset U_{m+1}$. Letting $\ovr{U}_i$ denote the image of $U_i$ in $M^\times \otimes \FFp$, we obtain an exact sequence of $\FFp$-modules
\[
\begin{tikzcd}
1\ar{r} & \ovr{U}_{m+1}\ar{r} & \ovr{U}_m\ar{r} & U_m/U_{m+1}\ar{r} & 1.
\end{tikzcd}
\]
The claim for these $m$ now follows by noting that 
\[
\left(- \otimes_{\FFp} \FFpb(\chi^{-1})\right)^{\Gal(M/K)}
\]
is exact since $\Gal(M/K)$ has order prime-to-$p$. In particular, it follows from this claim that $\dim_{\FFpb} \gr^m U_\chi$ equals $d_s$ for $1<s:=1+\frac{m}{e_M}<1+\frac{pe}{p-1}$ and $d_s$ as in Theorem \ref{thm:ddr3.1}.

Since the size of $B$ is equal to the $\FFpb$-dimension of $U_\chi$, it will suffice to prove that each graded is spanned by a subset of elements of $B$. By the dimensions of Theorem \ref{thm:ddr3.1}, we only need to consider integers $m$ such that $d_s>0.$

The rest of the proof follows from our constructions. If $m=\frac{epe_M}{p-1},$ then it suffices to note $u_\mathrm{cyc}$ is a non-trivial element of $\gr^m U_\chi=\Fil^m U_\chi,$ which is 1-dimensional.

If $0<m<\frac{epe_M}{p-1}$, then we may assume that $m=\frac{e_M \widetilde{m}}{p^f-1}=:\widetilde{m}'$ for some $\widetilde{m}\in W'$ as defined in \S\ref{subsec:u-ij}. The elements
\[
u_{\alpha}:=\varepsilon_{\pi^{\widetilde{m}'}}(\lambda_{\tau_{\alpha},\mu})
\]
for $\alpha=(\widetilde{m},k)\in W$ with $0\le k < f''$ are all non-trivial elements of $\gr^m U_\chi$ -- this follows, for example, since $E(x)\in 1+x+x^2\ZZp\llbracket x \rrbracket.$ It suffices to prove that these elements are linearly independent in $\gr^m U_\chi.$ Since $p\nmid m$, we showed that 
\[
\gr^m U_\chi \cong \left(U_m/U_{m+1}\otimes \FFpb(\chi^{-1})\right)^{\Gal(M/K)}.
\]
Since the map
\begin{align*}
l\otimes_{\FFp} \FFpb&\to U_m/U_{m+1}\otimes_{\FFp} \FFpb \\
a\otimes b &\mapsto (1+[a]\pi^m)\otimes b
\end{align*}
induced by $\varepsilon_{\pi^m}$ is an isomorphism and the elements 
\[
\{\lambda_{\tau_{\alpha},\mu} \mid \text{for }\alpha=(\widetilde{m},k)\text{ with }0\le k<f''\}
\]
are $\FFpb$-linearly independent in $l\otimes_{\FFp} \FFpb$, we see that their images are as well (recall that we defined $\tau_\alpha:=\tau_{i_{\widetilde{m}}+kf'}$, where $\widetilde{m}\equiv n_{i_{\widetilde{m}}}\bmod{(p^f-1)}$).

Lastly, if $m=0$, it suffices to note that $u_\mathrm{triv}=\pi\otimes 1$ does not lie in $\Fil^1U_\chi.$ 
\end{proof}

Given the basis $B$ above we obtain the dual basis $c_{\alpha}$ for $\alpha \in W$ with, additionally, $c_\mathrm{ur}$ and $c_\mathrm{tr}$ if $\chi$ is trivial or cyclotomic. This dual basis must, simply by duality, form a basis of $H^1(G_K,\FFpb(\chi))$ -- the subscripts refer to the unramified and tr\`es ramifi\'ee part of $H^1(G_K,\FFpb(\chi))$.

\begin{cor}\label{cor:defn-basis}
The set consisting of the classes $c_{\alpha}$ for $\alpha \in W$ with, additionally, $c_\mathrm{ur}$ if $\chi$ is trivial and $c_\mathrm{tr}$ if $\chi$ is cyclotomic forms a basis of $H^1(G_K,\FFpb(\chi))$.
\end{cor}

Once again we emphasise that the explicit basis as defined here may depend on the choice of extension $M$ and the choice of uniformiser $\pi$ of $M$ in \S\ref{subsec:fpb-dual}.

\section{Explicit Serre weights over ramified bases}\label{chap:proof}

In this section we use the explicit basis obtained in the previous section to give an explicit definition of the distinguished subspaces $L_V$ defined in \S\ref{sec:background-notation}: we will define them as a subspace spanned by a subset of our explicit basis elements. In Theorem \ref{thm:L_V=L_V^AH} we will prove the equivalence of our subspaces spanned by explicit basis elements with the distinguished subspaces $L_V$, thereby providing all the necessary theoretical foundation to give a new explicit reformulation of the conjectures on the modularity of mod $p$ Galois representations of totally real number fields. Results of \ref{sec:psBT-redns}, \ref{sec:restr_to_G_inf} and \ref{sec:expl_basis_L_V} can be seen as generalising results of \cite[\S3]{cegm17}.

\subsection{Reductions of pseudo-Barsotti Tate representations}\label{sec:psBT-redns}

One of the essential ingredients in our proof of the equivalence of the two formulations is the classification in \cite{gls15} of representations that arise as reductions of pseudo-Barsotti--Tate representations. Since we will heavily use their results and their notation, we will quickly recall their most important results adapted to the case at hand. We note that in \cite{gls15} only the case $p>2$ is covered. We will use Wang's extension \cite{wan17} of their results to $p=2$ for the remaining case.

\subsubsection{Field of norms and \'etale $\varphi$-modules}
Let us very briefly recall the theory of the field of norms and \'etale $\varphi$-modules. We are given a local field $K$ and a fixed uniformiser $\pi_K.$ We emphasise that our choices of uniformiser in this section and in the previous sections must be compatible. We define a compatible system of $p^n$th roots of our uniformiser as follows: let $\pi_0=-\pi_K$. We define $\pi_n$ inductively for any $n>0$ as a $p$th root of $\pi_{n-1}$, that is, satisfying $\pi_n^p=\pi_{n-1}$. Let $K_n:=K(\pi_n)$ and $K_\infty=\bigcup_{n=0}^\infty K_n.$ If $p=2$, then we assume that our uniformiser $\pi_K$ was chosen to satisfy \cite[Lem.~2.1]{wan17}; that is, we assume $K_\infty\cap K_{p^\infty}=K,$ where we define $K_{p^\infty}:=\bigcup_{n\ge 1} K(\zeta_{p^n})$ with $\zeta_{p^n}$ a primitive $p^n$-th root of unity. The theory of the field of norms lets us identify
\begin{align*}
k((u)) &\xrightarrow{\hspace{3mm}\sim\hspace{3mm}} \varprojlim_{N_{K_{n+1}/K_n}} K_n \\
u &\xmapsto{\hspace{9mm}} (\pi_n)_n,
\end{align*}
where the transition maps in the inverse limit are the norm maps. This construction extends to extensions of these fields and we get a natural isomorphism of absolute Galois groups $G_{k((u))}=G_{K_\infty}.$ On the other hand there exists an equivalence of abelian categories between the category of finite-dimensional $\FFpb$-representations of $G_{k((u))}$ and the category of \textbf{\'etale $\varphi$-modules}. These are finite $k((u))\otimes_{\FFp}\FFpb$-modules $\mathcal{M}$ equipped with a $\varphi$-semilinear map $\varphi\colon\mathcal{M}\to\mathcal{M}$ such that the induced $k((u))\otimes_{\FFp}\FFpb$-linear map $\varphi^*\mathcal{M}\to\mathcal{M}$ is an isomorphism. The functor $T$ from \'etale $\varphi$-modules to representations of $G_{k((u))}$ is given by
\[
T\colon\mathcal{M}\to \left(k((u))^\mathrm{sep}\otimes_{k((u))} \mathcal{M}\right)^{\varphi=1}.
\]
The isomorphism $k\otimes_{\FFp} \FFpb\to\prod_{i=0}^{f-1} \FFpb$ gives a decomposition $\mathcal{M}=\prod_i \mathcal{M}_i,$ where $\varphi$ now induces $\FFpb$-linear morphisms $\varphi\colon\mathcal{M}_{i-1}\to\mathcal{M}_i.$ We will make use of this theory by describing finite-dimensional $\FFpb$-representations of $G_{K_\infty}$ via \'etale $\varphi$-modules.

\subsubsection{The structure theorem of Gee--Liu--Savitt}\label{subsec:struc-thm-gls}
From now on let us fix a Serre weight $V=V_{\underline{\eta},\underline{0}}$ and a pair of characters $\chi_1,\chi_2:G_K\to \FFpb^\times$ as in \S\ref{sec:background-notation}. (Recall that we were free to assume that $V$ is of the form $V_{\underline{\eta},\underline{0}}$, since we can obtain the case of a general Serre weight from this one by twisting.) We will write $L_V$ for $L_{V_{\underline{\eta},\underline{0}}}(\chi_1,\chi_2)$ and we will assume this is non-empty. Recall that we defined the integers $r_i\in[1,p]$ via $r_i:=\eta_i+1$ and we write $\chi:=\chi_1\chi_2^{-1}.$ For any $0\le i<f$, we may decompose $\chi$ as a power of $\omega_i$ and an unramified character $\mu$ independent of $i.$ Then $\mu$ factors as a character $\mu:\Gal(L/K)\to\FFpb^\times$ for some finite unramified extension $L$ of $K$ of prime-to-$p$ order.

First we need to get rid of an exceptional pathological case. Later we will define the subspaces $L_V^{\mathrm{AH}}$ using our constructed basis and we will claim $L^{\mathrm{AH}}_V=L_V$ as defined above. Before we continue let us define this space in one exceptional case.

\begin{defn}\label{defn:exceptional-L_V^AH}
If $\chi$ is cyclotomic, $\chi_2$ is unramified and $r_i=p$ for all $0\le i<f,$ then we define
\[
L_V^{\mathrm{AH}}(\chi_1,\chi_2):=H^1(G_K,\FFpb(\chi)).
\]
\end{defn}
Note that it follows from the last part of the proof \cite[Thm.~6.1.8]{gls15} that if $\chi$ is cyclotomic, $\chi_2$ is unramified and $r_i=p$ for all $i$, then 
\[
L_V(\chi_1,\chi_2)=H^1(G_K,\FFpb(\chi)),
\]
so in this exceptional case we indeed have that $L_V^\mathrm{AH}(\chi_1,\chi_2)=L_V(\chi_1,\chi_2).$

We will from now on assume that if $\chi$ is cyclotomic and $\chi_2$ is unramified, then $r_i<p$ for at least one $i.$ This assumption has the consequence that the restriction map
\[
\mathrm{res}:H^1(G_K,\FFpb(\chi))\to H^1(G_{K_\infty},\FFpb(\chi))
\]
becomes injective when restricted to $L_V$ -- this is \cite[Lem.~5.4.2]{gls15} for $\chi$ non-cyclotomic and follows from the proof of \cite[Thm.~6.1.8]{gls15} when $\chi$ is cyclotomic. (Note that it follows from Wang \cite{wan17} that these results also work for $p=2$.) Alternatively, the injectivity for $\chi$ cyclotomic follows from the fact that under the assumption $L_V$ is contained in the \emph{peu ramifi\'e subspace} of $H^1(G_K,\FFpb(\chi))$ by \cite[Thm.~4.1.1]{ste20}. By \cite[Lem.~5.4.2]{gls15} the kernel of the restriction map is the tr\`es ramifi\'ee line spanned by the fixed uniformiser $-\pi_K$. Therefore, the restriction map is injective on this subspace. Since we are excluding the case above, we are free to restrict our representations to $G_{K_\infty}$ from now on and talk about the corresponding \'etale $\varphi$-modules.

\subsubsection{The integers \texorpdfstring{$t_i,s_i$}{t-i,s-i}}\label{subsubsec:defn-t-i-s-i} Before we state the main theorem of \cite{gls15}, we need to introduce just a little more notation: we need to introduce the integers $t_i,s_i$ corresponding to the maximal and minimal Kisin modules of \cite[\S5.3]{gls15}. We will take an ad hoc approach. Write $\chi_2|_{I_K}=\prod_i \omega_i^{m_i}$ for the uniquely determined integers $m_i\in[0,p-1]$ with not all $m_i$ equal to $p-1.$ Now let $\mathcal{S}$ be the set of $f$-tuples of non-negative integers $(a_0,\dotsc,a_{f-1})$ satisfying $\chi_2|_{I_K}=\prod_i \omega_i^{a_i}$ and $a_i\in[0,e-1]\cup[r_i,r_i+e-1]$ for all $i.$ This set is non-empty by our assumption that $L_V$ is non-empty. If we write $v_i$ for the $f$-tuple $(0,\dotsc,-1,p,\dotsc,0)$ with the $-1$ in the $i$th position and $v_{f-1}$ for $(p,0,\dotsc,0,-1)$, then it is proved in \cite[Lem.~5.3.1]{gls15} that there exists a subset $J\subseteq\{0,\dotsc,f-1\},$ possibly empty, such that
\[
(m_0,\dotsc,m_{f-1})+\sum_{i\in J} v_i \in \mathcal{S}
\]
and we can choose $J=J_\mathrm{min}$ minimal in the sense that it is contained in any other subset satisfying this requirement. Then we define
\[
(t_0,\dotsc,t_{f-1}):=(m_0,\dotsc,m_{f-1})+\sum_{i\in J_\mathrm{min}} v_i \in \mathcal{S}.
\]
We will define $s_i:=(r_i+e-1)-t_i$ for all $i.$ Note that it follows from our construction that $s_i,t_i\in[0,e-1]\cup[r_i,r_i+e-1]$ and $s_i+t_i=r_i+e-1$ for all $i.$ 

\subsubsection{The main structure theorem of Gee--Liu--Savitt} We will understand $(a)_i$ to mean the following: $(a)_i=a$ if $i\equiv 0 \bmod f$ and $(a)_i=1$ otherwise. Now we are ready to state the main structure theorem of \cite{gls15}.

\begin{thm}\label{thm:gls}
The image of the subspace $L_V$ of $H^1(G_K,\FFpb(\chi))$ after restriction to $H^1(G_{K_\infty},\FFpb(\chi))$ consists precisely of those classes that can be represented by \'etale $\varphi$-modules $\mathcal{M}$ of the following form: we can choose bases $e_i,f_i$ of $\mathcal{M}_i$ for which $\varphi$ has the form
\begin{align*}
\varphi(e_{i-1}) &= u^{t_i}e_i; \\
\varphi(f_{i-1}) &= (a)_i u^{s_i}f_i+y_ie_i,
\end{align*}
where $a=\mu(\Frob_K)$ and $y_i\in \FFpb\llbracket u\rrbracket$ is a polynomial of which the degrees of the non-zero terms lie in the interval

\begin{itemize}
\item $[0,s_i-1]$ if $t_i\ge r_i$ except, possibly, if $\chi$ is trivial;
\item $\{t_i\}\cup [r_i,s_i-1]$ if $t_i<r_i$ except, possibly, if $\chi$ is trivial.
\item If $\chi$ is trivial, we make any single choice of $i_0\in[0,\dotsc,f-1]$ and for $i_0$ in addition to the intervals defined above we also allow terms of degree
\[
\left\{s_{i_0}+\frac{1}{p^f-1}\sum_{k=1}^{f} p^{f-k}(s_{{i_0}+k}-t_{{i_0}+k}) \right\}.
\]
\end{itemize}
In every case the polynomials $y_i$ are uniquely determined by $\mathcal{M}$ and every collection of polynomials subject to the conditions above corresponds to a class in $L_V.$
\end{thm}

\begin{proof} 
For $p>2$, the statement of \cite[Thm.~5.1.5]{gls15} shows that the corresponding Kisin modules (which are just lattices in $\mathcal{M}$) have the form above. It follows from \cite[\S5]{gls15} in the non-cyclotomic case and \cite[\S6]{gls15} in the cyclotomic case that we get the full subspace $L_V$, since we chose the $t_i$ and $s_i$ to be minimal and maximal, respectively. In the non-cyclotomic case it follows from the proof of \cite[Thm.~5.4.1]{gls15} that the corresponding elements of $H^1(G_{\Kinf},\FFpb(\chi))$ (and hence the corresponding \'etale $\varphi$-modules) are uniquely determined by $y_i$ and that any collection of polynomials $\{y_i\}$ satisfying the conditions above corresponds to a class in $L_V.$ In the cyclotomic case (under the assumption that we are not in the exceptional case above) the proof of \cite[Thm.~6.1.8]{gls15} gives the same uniqueness and existence result. It follows from Wang \cite{wan17} that the results from \cite{gls15} quoted here and their proofs carry over naturally to $p=2$ (assuming we have carefully chosen our uniformiser of $K$) giving the required result in this case as well.
\end{proof}

\begin{defn}\label{defn:interval_sets}
Given $\{r_i,t_i,s_i\mid 0\le i <f\}$ as defined at the beginning of \S\ref{subsec:struc-thm-gls} and in \S\ref{subsubsec:defn-t-i-s-i}, we define
\begin{itemize}
\item $\mathcal{I}_i:=[0,s_i-1]$ if $t_i\ge r_i;$
\item $\mathcal{I}_i:=\{t_i\}\cup [r_i,s_i-1]$ if $t_i<r_i.$
\end{itemize}
\end{defn}
Here we have simply given the intervals appearing in the theorem a name for later convenience. We note that (both in the definition and  the theorem) we follow standard conventions and let $\mathcal{I}_i:=\varnothing$ if $t_i\ge r_i$ and $s_i=0,$ whereas, if $t_i<r_i$ and $s_i=r_i,$ we let $\mathcal{I}_i:=\{t_i\}.$

\subsection{Restricting to \texorpdfstring{$G_{M_\infty}$}{G-M-infty}}\label{sec:restr_to_G_inf}
In this subsection we will use the structure theorem of \cite{gls15} from the previous subsection to find the image of $L_V$ (after restricting to $G_{M_\infty}$) in terms of Artin--Schreier theory. Later we will use an explicit reciprocity law (see Theorem \ref{thm:schmidt-rec}) to compare the Artin--Schreier theoretic image of $L_V$ to our explicit basis elements of the previous section.

Recall that we let $M=L(\pi)$ be a totally tamely ramified extension of an unramified extension $L/K$ of degree prime-to-$p$, where the uniformiser $\pi$ of $M$ is a solution of $x^{e_M}+\pi_K=0$ for $e_M\mid p^f-1$ and we take $M$ sufficiently large so that $\chi|_{G_M}$ is trivial. We will denote the residue field of $M$ by $l$. 

Since $e_M$ and $p$ are coprime, we have that for each $n\ge 1$ there is a unique $p^n$th root $\pi^{1/p^n}$ of $\pi$ that satisfies $(\pi^{1/p^n})^{e_M}=\pi_n,$ i.e. the choice of $p^n$th root of $\pi$ is compatible with the choice of $p^n$th root of $-\pi_K$ in the definition of $K_\infty.$ Let $M_n:=M(\pi^{1/p^n})$ and $M_\infty=\bigcup_{n\ge0} M_n.$ It follows by the theory of the field of norms and local class field theory applied to $l((u))$ that
\[
\Gal\left(M_\infty^{(p)}/M_\infty\right)\simeq \Gal\left(l((u))^{(p)}/l((u))\right)\simeq l((u))^\times \otimes \FFp,
\]
where, for any field $F$, we write $F^{(p)}$ for the maximal exponent $p$ abelian extension of $F$. Furthermore, if we let $\psi:l((u))^{\mathrm{sep}}\to l((u)^{\mathrm{sep}}$ denote the Artin--Schreier map defined by $\psi(x)=x^p-x$, then Artin--Schreier theory gives an isomorphism
\begin{equation}\label{eqn:martsch}
H^1(G_{\Minf},\FFpb)=H^1(G_{l((u))},\FFpb)\cong (l((u))/\psi l((u)))\otimes_{\FFp} \FFpb.
\end{equation}
By \cite[Ch.~X,~\S3]{ser79} the element $a\in l((u))$ concretely corresponds to the homomorphism $f_a:G_{l((u))}\to \FFp$ given by $f_a(g)=g(x)-x,$ where $x\in l((u))^\mathrm{sep}$ is chosen so that $\psi(x)=a.$ Any two-dimensional $\FFpb$-representation $\rho$ of $G_{l((u))}$ that arises as the extension of a character $\chi$ by itself corresponds naturally to a cocyle in $H^1(G_{l((u))},\FFpb)$ via the standard isomorphism
\begin{equation}\label{eqn:ext-h1}
\begin{aligned}
\mathrm{Ext}_{\mathrm{Rep}_{\FFpb}(G_{l((u))})}(\chi,\chi) &\to H^1(G_{l((u))},\FFpb), \\
\rho  \cong \chi\otimes \begin{pmatrix} \mathbf{1} & c_\rho \\ 0 & \mathbf{1} \end{pmatrix}&\mapsto [\sigma  \mapsto c_\rho(\sigma)].
\end{aligned}
\end{equation}

Since $M/K$ is an extension of degree prime-to-$p$, so is $M_\infty/K_\infty.$ Therefore, the restriction map $H^1(G_{K_\infty},\FFpb(\chi))\to H^1(G_{M_\infty},\FFpb)$ is injective and we may describe $L_V$ completely by its image in $H^1(G_{M_\infty},\FFpb)$ or by the classes corresponding to this image via Artin--Schreier theory, which is what we will do.

If $\mathcal{M}$ is an \'etale $\varphi$-module with corresponding $G_{\Kinf}$-representation denoted by $T(\mathcal{M}),$ then it is easy to check that the \'etale $\varphi$-module corresponding to $T(\mathcal{M})|_{G_{M_\infty}}$ is
\[
\mathcal{M}_M:=l((u))\otimes_{k((u)),u\mapsto u^{e_M}} \mathcal{M}.
\]
Applying this to one of the \'etale $\varphi$-modules arising in the statement of Theorem \ref{thm:gls}, it follows that (with the obvious choice of basis $e_i,f_i$ for $\mathcal{M}_{M,i}$) the matrix of $\varphi:\mathcal{M}_{M,i-1}\to\mathcal{M}_{M,i}$ is
\[
\begin{pmatrix}
u^{t_i e_M} & y_{i,M} \\
0 & (a)_i u^{s_i e_M}
\end{pmatrix},
\]
where we let $y_{i,M}\in\FFpb[u^{e_M}]$ denote the image of $y_i\in\FFpb[u]$ under $u\mapsto u^{e_M}.$ Here the $\mathcal{M}_{M,i}$ are periodic with period $f[l:k],$ but $t_i,s_i$ and $y_{i,M}$ only depend on $i\bmod f.$

Recall that we defined $a=\mu(\Frob_K)$ and that the unramified character $\mu$ factors as $\mu\colon\Gal(L/K)\to\FFpb^\times,$ so that $a^{[l:k]}=1.$ We will try to find a change of basis such that the diagonal entries of $\varphi:\mathcal{M}_{M,i-1}\to\mathcal{M}_{M,i}$ are identical. Since
\[
\restr{\rho}{I_K}\sim \begin{pmatrix} \prod_i \omega_i^{s_i} & * \\ 0 & \prod_i \omega_i^{t_i}\end{pmatrix}
\]
and $\chi_1^{e_M}=\chi_2^{e_M},$ it follows that $\left(\frac{p^f-1}{e_M}\right)$ divides
\[
\sum_{j=0}^{f-1}(s_{i+1+j}-t_{i+1+j})p^{f-1-j}
\]
for all $i.$ Hence, we may define integers
\[
\alpha_i:=\left(\frac{e_M}{p^f-1}\right)\sum_{j=0}^{f-1}(s_{i+1+j}-t_{i+1+j})p^{f-1-j}.
\]
If we set $e_i':=u^{\alpha_i}e_i$ and $f_i'=a^{\lfloor i/f\rfloor} f_i$ for $0\le i \le f[l:k]-1,$ then after this change of basis the matrix of $\varphi:\mathcal{M}_{M,i-1}\to\mathcal{M}_{M,i}$ becomes
\[
\begin{pmatrix}
u^{t_i e_M+p\alpha_{i-1}-\alpha_i} & a^{\lfloor{i-1/f}\rfloor}y_{i,M}u^{-\alpha_i} \\
0 & u^{s_i e_M}
\end{pmatrix}.
\]
That is, upon noting that $t_i e_M+p\alpha_{i-1}-\alpha_i=s_ie_M$, after the change of basis the matrix of $\varphi:\mathcal{M}_{M,i-1}\to \mathcal{M}_{M,i}$ becomes
\[
u^{s_ie_M}\otimes\begin{pmatrix}
1 & a^{\lfloor i-1/f\rfloor}y_{i,M}u^{-(s_ie_M+\alpha_i)} \\
0 & 1
\end{pmatrix}
\] 
for $0 \le i \le f[l:k]-1.$ The exponents of $u$ will be important later, so let us give them a name.

\begin{defn}
For all $0\le i <f,$ we define the constants
\begin{align*}
\xi_i:=(p^f-1)s_i+\sum_{j=0}^{f-1}(s_{i+j+1}-t_{i+j+1})p^{f-1-j}.
\end{align*}
\end{defn}

Note that for the exponent in the matrix above we have that
\[
s_ie_M+\alpha_i=\frac{e_M\xi_i}{p^f-1}.
\]
In line with our earlier notation, let us therefore write $\xi_i'$ for $\frac{e_M\xi_i}{p^f-1}.$

\begin{lem}\label{lem:xi-i-congruence}
We have that $\xi_i\equiv n_i \bmod{p^f-1}$ and, therefore, that $\xi_i'\equiv \frac{e_Mn_i}{p^f-1}\bmod e_M$ for all $i.$
\end{lem}
\begin{proof}
From $\restr{\chi}{I_K}=\restr{\chi_1\chi_2^{-1}}{I_K}$ and Theorem \ref{thm:gls} above, we have that $\prod_i \omega_i^{s_i-t_i}=\omega_i^{n_i},$ that is
\[
\sum_{j=0}^{f-1}(s_{i+1+j}-t_{i+1+j})p^{f-1-j}\equiv n_i \bmod{p^f-1}.
\]
But this sum is equal to $\xi_i-s_i(p^f-1)$.
\end{proof}

Recall that in Section \ref{subsec:u-ij} we proved that for any embedding $\tau:k\to\FFpb$ we have that $l\otimes_{k,\tau} \FFpb$ is free of rank one over $\FFpb[\Gal(L/K)].$ It follows that the $\mu$-eigenspace $\Lambda_{\tau,\mu}$ defined as
\[
\left\{a\otimes b\in l\otimes_{k,\tau}\FFpb\;\middle|\; g(a)\otimes b =(1\otimes\mu(g))a\otimes b\text{ for all }g\in\Gal(L/K)\right\}
\]
is 1-dimensional over $\FFpb.$ For $a=\mu(\Frob_K),$ we define the non-zero element 
\[
\lambda_{\tau,\mu}:=(1,a^{-1},\dotsc,a^{1-[l:k]})\in l\otimes_{k,\tau} \FFpb
\]
and note that this is in fact a basis of the one-dimensional vector space $\Lambda_{\tau,\mu}.$ Here we have used the standard identification
\begin{align*}
l\otimes_{k,\tau} \FFpb &\xrightarrow{\hspace{3mm}\sim\hspace{3mm}} \prod_{\sigma_i} \FFpb; \\
x\otimes y &\xmapsto{\hspace{8mm}} (\sigma_i(x)y)_i,
\end{align*}
where the sum runs over all $k$-embeddings $\sigma_i\colon l\hookrightarrow \FFpb$ such that $\restr{\sigma_i}{k}=\tau.$ Similarly, $\lambda_{\tau,\mu^{-1}}:=(1,a,\dotsc,a^{[l:k]-1})\in l\otimes_{k,\tau} \FFpb$ spans $\Lambda_{\tau,\mu^{-1}}.$ Using this notation we may define the element
\[
(a^{\lfloor i/f\rfloor})_{i=0,\dotsc,f[l:k]-1}=(\lambda_{\overline{\tau}_i,\mu^{-1}})_{i=0,\dotsc,f-1}\in l\otimes_{\FFp}\FFpb.
\]

We can consider $\mathcal{M}_M$ as an $l((u))\otimes_{\FFp}\FFpb$-module with the obvious choice of basis. Then the matrix of $\varphi_{\mathcal{M}_M}$ becomes
\[
(u^{s_ie_M})_{i=0,\dotsc,f-1}\otimes
\begin{pmatrix}
1 & (a^{-1}\lambda_{\overline{\tau}_i,\mu^{-1}}y_{i,M}u^{-\xi_i'})_{i=0,\dotsc,f-1} \\
0 & 1
\end{pmatrix}.
\]
By the equivalence of categories $T$ between \'etale $\varphi$-modules and $\FFpb$-representations of $G_{l((u))}$ combined with Equations (\ref{eqn:martsch}) and (\ref{eqn:ext-h1}), this \'etale $\varphi$-module then corresponds to a class in $l((u))/\psi l((u))\otimes_{\FFp} \FFpb.$ To understand which class exactly, we prove the following proposition.

\begin{prop}
Suppose $\mathcal{M}$ is an \'etale $\varphi$-module over $l((u))$ such that $\varphi_{\mathcal{M}}:\mathcal{M}\to\mathcal{M}$ is given by the matrix
\[
\begin{pmatrix}
1 & z \\
0 & 1
\end{pmatrix}.
\]
Then the image of $\mathcal{M}$ under the isomorphism 
\[
\mathrm{Ext}^1_{\mathrm{Rep}_{\FFp}\left(G_{l((u))}\right)}\left(1,1\right)\cong l((u))/\psi l((u))
\]
is given by (the class of) $z.$
\end{prop}

\begin{proof}
This is a standard argument; see, for example, \cite[Prop.~6.2.3]{ste20}.
\end{proof}

After extending coefficients to $\FFpb,$ it follows from the proposition that the Artin--Schreier extension of $l((u))$ corresponding to $T(\mathcal{M}_M)$ above is determined by the element $(\lambda_{\overline{\tau}_i,\mu^{-1}}y_{i,M}u^{-\xi_i'})_{i=0,\dotsc,f-1}$ -- note the absence of $a^{-1}$ because scaling by an element of $\FFpb^\times$ does not change the extension. Therefore, we have proved the following result.

\begin{cor}\label{cor:im-L-V}
The image of $L_V$ in $H^1(G_{M_\infty},\FFpb)=\Hom(G_{l((u))},\FFpb)$ is spanned by the homomorphisms $f_{y_{i,M}\lambda_{\overline{\tau}_i,\mu^{-1}}u^{-\xi_i'}}$ corresponding via Artin--Schreier theory to the elements
\[
y_{i,M}\lambda_{\overline{\tau}_i,\mu^{-1}}u^{-\xi_i'} \in l((u))\otimes_{k,\overline{\tau}_i}\FFpb\subseteq l((u))\otimes_{\FFp} \FFpb.
\]
More precisely, the image is spanned by homomorphisms corresponding to
\[
\lambda_{\overline{\tau}_i,\mu^{-1}}u^{de_M-\xi_i'} \in l((u))\otimes_{\FFp} \FFpb
\]
for all $0\le i <f$ and all $d\in\mathcal{I}_i,$ together with the class $\lambda_{\overline{\tau}_{i_0},\mu^{-1}}u^{d_{i_0}e_M - \xi_{i_0}'}$ for 
\[
d_{i_0}:=s_{i_0}+\frac{1}{p^f-1}\sum_{k=0}^{f-1} (s_{{i_0}+1+k}-t_{{i_0}+1+k})p^{f-1-k}
\] 
if $\chi$ is trivial.
\end{cor}

\subsection{An explicit basis of \texorpdfstring{$L_V$}{L-V}}\label{sec:expl_basis_L_V}

In this subsection we would like to compare our description of $L_V$ above to the previously constructed basis elements $c_{\alpha}$ for $\alpha\in W$ and give an explicit description of $L_V$ in this way. To explain our approach we need to introduce some theory first.

We have a pairing $\langle\cdot,\cdot\rangle: H^1(G_{\Minf},\FFpb)\times \Gal(\Minf^{(p)}/\Minf) \to \FFpb$ given by evaluation. Recall that the theory of the field of norms and the local Artin map allowed us to write $\Gal(\Minf^{(p)}/\Minf)\simeq l((u))^\times\otimes \FFp$ and that we used Artin--Schreier theory to write $H^1(G_{\Minf},\FFpb)\simeq \left(l((u))/\psi l((u))\right)\otimes_{\FFp} \FFpb.$ The real advantage is that we can make the pairing explicit now.

\begin{thm}\label{thm:schmidt-rec}
Let $\sigma_b\in \Gal(\Minf^{(p)}/\Minf)$ be the Galois element corresponding via the local Artin map to an element $b\in l((u))^\times\otimes \FFp,$ and let $f_a$ be the element of $H^1(G_{\Minf},\FFpb)$ corresponding to $a \in (l((u))/\psi l((u)))\otimes_{\FFp} \FFpb$ via Artin--Schreier theory. Then
\[
\langle f_a, \sigma_b \rangle = \mathrm{Tr}_{l\otimes_{\FFp}\FFpb/\FFpb}\left ( \mathrm{Res}(a\cdot \frac{db}{b})\right ),
\]
where $\mathrm{Res}(f du)$ simply means the coefficient of $u^{-1}$ in $f$.
\end{thm}

\begin{proof} See \cite[XIV,~Cor.~to~Prop.~15]{ser79}. \end{proof}

Since we already have a description of the image of $L_V$ in $H^1(G_{\Minf},\FFpb)$ and we defined $c_{\alpha}\in H^1(G_K,\FFpb(\chi)),$ for $\alpha \in W,$ as the basis dual to basis elements 
\[
u_{\alpha}\in U_\chi:=\left(M^\times \otimes \FFpb(\chi^{-1})\right)^{\Gal(M/K)}\subseteq M^\times\otimes \FFpb,
\]
we may hope to get a description of the elements $u_{\alpha}$ in $\Gal(\Minf^{(p)}/\Minf)$ and look to find conditions under which the pairing of $u_{\alpha}$ with $L_V$ vanishes. The local Artin map $\mathrm{Art}_M^{-1}$ induces an isomorphism $\Gal(M^{(p)}/M)\simeq M^\times \otimes \FFp,$ so by extension by scalars to $\FFpb$ we get $\mathrm{Art}_M(u_{\alpha})\in \Gal(M^{(p)}/M)\otimes_{\FFp} \FFpb.$

\begin{lem}\label{lem:pair-comp}
We have a commutative diagram of pairings
\[
\begin{tikzcd}
H^1(G_{\Minf},\FFpb) &[-30pt]\times &[-30pt]\Gal(\Minf^{(p)}/\Minf)\ar{d}{\mathrm{pr}} \ar{r} & \FFpb \\[10pt]
H^1(G_{M},\FFpb)\ar{u}{\mathrm{res}}&\times &\Gal(M^{(p)}/M) \ar{r} & \FFpb
\end{tikzcd}
\]
in the sense that $\langle\mathrm{pr}(\alpha),\beta\rangle=\langle\alpha,\mathrm{res}(\beta)\rangle,$ where the pairings are given by evaluation, the map $\mathrm{res}$ is given by restriction to $G_{\Minf}$ and $\mathrm{pr}$ is given by restricting an automorphism to $M^{(p)}.$ 
\end{lem}

\begin{proof} Since $H^1(G_{M},\FFpb)=\Hom(G_M,\FFpb)$ (and similarly for $\Minf$) and the pairings are given by evaluation, this follows from the definitions.\end{proof}

By the lemma above it is enough to consider the pairing of $L_V$ with elements of $\Gal(\Minf^{(p)}/\Minf)$ that map to the elements $\mathrm{Art}_M(u_{\alpha})$ under $\mathrm{pr}.$ To be able to use the explicit reciprocity law of Theorem \ref{thm:schmidt-rec}, we need to use the local class field theory of $l((u))$ and the theory of the field of norms to get an isomorphism $\Gal(\Minf^{(p)}/\Minf)\simeq l((u))^\times \otimes \FFp.$ From the field of norms identification $l((u))\simeq\varprojlim_{N_{M_{n+1}/M_n}} M_n$ we get a natural map $\mathrm{pr}_M:l((u))\to M$ by projecting onto the $M$-component which is compatible with local class field theory in the following natural way.

\begin{lem}
The following diagram commutes
\[
\begin{tikzcd}
\Gal(\Minf^{(p)}/\Minf) \ar{d}{\mathrm{pr}}\ar{r}{\sim}[swap]{f.o.n.} & \Gal\left(l((u))^{(p)}/l((u))\right)\ar{r}{\sim}[swap]{\mathrm{Art}^{-1}_{l((u))}} & l((u))^\times\otimes\FFp \ar{d}{\mathrm{pr}_M} \\
\Gal(M^{(p)}/M)\ar{rr}{\sim}[swap]{\mathrm{Art}^{-1}_M} & & M^\times\otimes \FFp.
\end{tikzcd}
\]
\end{lem}

\begin{proof} See \cite[Lem.~3.1.5]{cegm17}.\end{proof}

The missing step is to find elements that project to our previous $u_{\alpha}$ under $\mathrm{pr}_M.$ Recall that $E(X)\in\ZZp\llbracket X \rrbracket$ denotes the Artin-Hasse exponential. Then, for each $r\ge 1,$ we get a homomorphism
\begin{align*}
\varepsilon_{u^r}:l\otimes_{\FFp} \FFpb &\to l((u))^\times\otimes_{\ZZ} \FFpb \\
a\otimes b &\mapsto E(au^r)\otimes b
\end{align*}
and, analogously to our earlier definitions of $u_{\alpha}$ in \S\ref{subsec:u-ij}, we will define
\[
\tilde{u}_{\alpha}:=\varepsilon_{u^{m'}}(\lambda_{\tau_{\alpha},\mu})
\]
for $\alpha=(m,k)\in W.$

\begin{prop}
Under the map $\mathrm{pr}_M:l((u))^\times \otimes \FFp \to M^\times \otimes \FFp,$ we have that 
\[
\mathrm{pr}_M\left(E(au^r)\right)=E([a]\pi^r)
\]
for any $r\ge 1$ which is prime to $p.$ In particular, after extending by scalars to $\FFpb,$ we have that $\mathrm{pr}_M(\tilde{u}_{\alpha})=u_{\alpha}$ for $\alpha\in W.$
\end{prop}

\begin{proof}
To see why this is true, we recall that addition in $\varprojlim M_n$ is defined by the formula
\[
(x+y)^{(n)}=\lim_{m\to \infty} N_{M_{n+m}/M_n}(x^{(n+m)}+y^{(n+m)}).
\]
If we write $E(X)=\sum_{k\ge 0} c_k X^k\in \ZZp\llbracket X \rrbracket,$ then we see that
\begin{align*}
\overline{E}(au^r)&=\overline{E}\left(([a]\pi^r,([a]\pi^r)^{1/p},([a]\pi^r)^{1/p^2},\dotsc)\right) \\
&= \sum_{k\ge 0} \overline{c}_k\left([a]\pi^r,([a]\pi^r)^{1/p},([a]\pi^r)^{1/p^2},\dotsc\right)^k \\
&\mapsto \lim_{m\to \infty} N_{M_m/M}(\sum_{k\ge 0} c_k([a]\pi^r)^{k/p^m}),
\end{align*}
where we have taken the projection onto the $M$-component in the last line. Now the proposition follows immediately from \cite[Lem.~3.5.1]{cegm17}, where it is proved that for any $n>1$, $a\in l$ and $r\ge 1$ such that $(r,p)=1$ we have
\[
N_{M_n/M}E([a^{1/p^n}](\pi^{1/p^n})^r)=E([a]\pi^r).
\]
The second claim follows since for all $\alpha=(m,k)\in W$ we have that $m'$ is coprime to $p.$
\end{proof}

Now we have covered all the theory to state and prove our explicit version of the space $L_V$ in terms of the basis elements constructed previously for $H^1(G_K,\FFpb(\chi)).$ Recall that we defined the indexing set $W$ of all basis elements in \S\ref{subsec:u-ij}. In the following definition we will give the definition of an indexing set $J_V^\mathrm{AH}$ as a subset of $W$ such that the span of basis elements $c_{\alpha}$ over this indexing set will give us $L_V.$

\begin{defn}\label{defn:J-V}
We let $J^\mathrm{AH}_V(\chi_1,\chi_2)$ denote the subset of all $\alpha=(k,m)\in W$ such that there exist integers $0\le i<f,$ $d\in\mathcal{I}_i$ and $j\ge 0$ such that
\begin{enumerate}[\hspace{6mm} 1)]
\item $p^j m'=\xi_i'-de_M$ and
\item $i_m+kf'\equiv i-j\bmod f,$
\end{enumerate}
where $i_m\in\{0,\dotsc,f'-1\}$ is defined by $m\equiv n_{i_m} \bmod{(p^f-1)}.$
\end{defn}

Note that, even though we need to calculate the integers $s_i,t_i$ corresponding to minimal and maximal Kisin modules to be able to define $\xi_i$ and $\mathcal{I}_i$ concretely, these integers can be calculated completely explicitly (as we have done above) without the use of any $p$-adic Hodge theory. The dependence on $\{r_i\}$ or, equivalently, the dependence on the Serre weight $V_{\underline{\eta},\underline{0}}$ for $\eta_i=r_i-1$ is implicit in the definitions of $s_i,t_i$ (hence $\xi_i$) and $\mathcal{I}_i$.

\begin{prop}\label{prop-dim-comb}
$\left\vert{J^\mathrm{AH}_V(\chi_1,\chi_2)}\right\vert \le \sum_{i=0}^{f-1} \left\vert{\mathcal{I}_i}\right\vert.$
\end{prop}

\begin{proof}
By multiplying the left and right hand sides of the first equation in Definition \ref{defn:J-V} by $(p^f-1)/e_M,$ we obtain the equation 
\begin{equation}\label{eqn:first-eqn-defn-6.3.5}
p^j m=\xi_i-d(p^f-1);
\end{equation}
this follows from the definitions of $m'$ and $\xi_i'$. Therefore, given an $\alpha=(m,k)\in W$, we can solve the first equation of Definition \ref{defn:J-V} for some $i,d$ and $j$ if and only if we can solve Equation \ref{eqn:first-eqn-defn-6.3.5} for the same $i,d$ and $j.$ It follows that we may assume without loss of generality that $e_M=p^f-1.$

For a fixed choice of $0\le i<f$ and $d\in \mathcal{I}_i$, we set $j:=v_p(\xi_i-d(p^f-1))$ and define
\[
a:=\frac{\xi_i-d(p^f-1)}{p^j}.
\]
It follows from Lemma \ref{lem:xi-i-congruence} that $p^j a\equiv n_i \bmod{(p^f-1)}$, so \[
a\equiv n_{i-j}\bmod{(p^f-1)}.
\] 
Therefore, by definition of $W'$ in \S\ref{subsec:u-ij}, the first equation of Definition \ref{defn:J-V} will have a solution for some $m\in W'$ if and only if $0<a<\frac{pe}{p-1}(p^f-1)$. Suppose that this is the case. Then we have a unique $m\in W'$ solving the equation, namely $m=a.$ Since the integers $m\in W'$ are distinct, this means that $\alpha=(m,k)\in W$ solving the first equation of Definition \ref{defn:J-V} for our fixed choice of $(i,d)$ exists and is unique up to $k\in\{0,\dotsc,f''-1\}$. 

Note that it follows from $m\equiv n_{i_m}\bmod{(p^f-1)}$ and $m\equiv n_{i-j}\bmod{(p^f-1)}$ that $i_m\equiv i-j \bmod{f'}.$ The second equation of Definition \ref{defn:J-V} requires that $k\equiv \frac{i-j-i_m}{f'}\bmod{f''}$, thereby determining  $k\in\{0,\dotsc,f''-1\}$ uniquely. In other words, we have proved that there exists a unique $\alpha\in J^\mathrm{AH}_V(\chi_1,\chi_2)$ corresponding to our fixed choice of $(i,d)$ if
\[
0<\frac{\xi_i-d(p^f-1)}{p^j}<\frac{ep}{p-1}(p^f-1)
\]
and no $\alpha$ corresponding to $(i,d)$ otherwise.
\end{proof}

We remark that it is possible to show with a direct argument, which carefully uses the minimality (resp. maximality) of the integers $t_i$ (resp. $s_i$), that the inequalities on $a$ in the proof above are always satisfied. This also follows as a consequence of Theorem \ref{thm:L_V=L_V^AH} which implies that $\left\vert{J^\mathrm{AH}_V(\chi_1,\chi_2)}\right\vert = \sum_{i=0}^{f-1} \left\vert{\mathcal{I}_i}\right\vert.$

Let $c_{\alpha}$ denote the basis elements of $H^1(G_K,\FFpb(\chi))$ defined in \S\ref{sec:expl-basis}, where we assume we define these making the same choices in the definition of $M=L(\pi)$ as in the current section, that is, they are defined using the same fixed uniformiser $\pi_K\in K$, the same unramified extension $L$ of $K$ and the same $\pi$ satisfying $\pi^{e_M}+\pi_K=0.$

\begin{defn}\label{defn:L_V^AH}
We define $L_V^\mathrm{AH}(\chi_1,\chi_2)$ to be the span of
\[
\{c_{\alpha} \mid \alpha\in J^\mathrm{AH}_V(\chi_1,\chi_2)\}
\]
together with $c_\mathrm{un}$ if $\chi$ is trivial and $c_\mathrm{tr}$ if $\chi$ is cyclotomic, $\chi_2$ unramified and $r_i=p$ for all $i.$
\end{defn}

\begin{cor} \label{cor:L-V-dims}
$\dim_{\FFpb} L_V^{\mathrm{AH}}(\chi_1,\chi_2)\le\dim_{\FFpb} L_V(\chi_1,\chi_2).$
\end{cor}

\begin{proof}
This follows immediately from the definitions, Proposition \ref{prop-dim-comb} and Theorem \ref{thm:gls} (since $y_i$ uniquely determines $\mathcal{M}$ and vice-versa).
\end{proof}

\begin{thm}\label{thm:L_V=L_V^AH}
$L_V(\chi_1,\chi_2) = L_V^{\mathrm{AH}}(\chi_1,\chi_2).$
\end{thm}

\begin{proof}
By Corollary \ref{cor:L-V-dims} it is enough to prove $L_V \subseteq L_V^{\mathrm{AH}}.$ We defined $L_V^{\mathrm{AH}}$ in terms of the basis elements $c_{\alpha}$ dual to $u_{\alpha},$ so we need to prove that the image of every class of $L_V$ in $H^1(G_M,\FFpb)$ is orthogonal to the elements $\mathrm{Art}_M(u_{\alpha})$ in $\Gal(M^{(p)}/M)$ for $\alpha\notin J_V^{\mathrm{AH}}$ under the pairing of Lemma \ref{lem:pair-comp}. If $\chi$ is cyclotomic, we also need to prove orthogonality under pairing with $u_\mathrm{cyc}$, which we will do first.

If $\chi$ is cyclotomic, then the classes that are orthogonal to $u_\mathrm{cyc}$ are exactly the peu ramifi\'ees classes. We see this, for example, because the subspace orthogonal to $u_\mathrm{cyc}$ is spanned by the basis elements $c_{\alpha}$ for $\alpha\in W$ and $c_\mathrm{ur}$ if $\chi$ is trivial. Since $c_\mathrm{tr}$ spans the one-dimensional subspace defined as $\Fil^{1+\frac{ep}{p-1}} H^1(G_K,\FFpb(\chi)),$ it follows that these basis elements span the space $\mathrm{Fil}^{<1+\frac{ep}{p-1}} H^1(G_K,\FFpb(\chi)),$ which equals the peu ramifi\'e subspace by \cite[Cor.~4.4.1]{ste20}. However, since we excluded the exceptional case in which $L_V$ equals $H^1(G_K,\FFpb)$, we have that $L_V$ is always contained in the peu ramifi\'e subspace of $H^1(G_K,\FFpb)$ by \cite[Thm.~4.1.1]{ste20}. Therefore, $L_V$ is always contained in the subspace orthogonal to $u_\mathrm{cyc}.$
 
As outlined above we apply the explicit reciprocity law of Theorem \ref{thm:schmidt-rec} to the elements $\tilde{u}_{\alpha}$ and the elements of Corollary \ref{cor:im-L-V}. We must show that for all $0\le i <f$, $d\in \mathcal{I}_i$ and $\alpha=(m,k)\notin J_V^{\mathrm{AH}}$
\[\mathrm{Tr}_{l\otimes_{\FFp} \FFpb/\FFpb} \mathrm{Res}\left(\mathrm{dlog}(\tilde{u}_{\alpha})\cdot \lambda_{\tau_i,\mu^{-1}} u^{de_M-\xi_i'}\right)=0
\]
and if $\chi=1$ we also need to show that the pairing with $\lambda_{\overline{\tau}_{i_0},\mu^{-1}}u^{d_{i_0}e_M-\xi_{i_0}'}$ vanishes. Since
\[
\mathrm{dlog} E(X)=(X+X^p+X^{p^2}+\dots)\mathrm{dlog}X
\]
and $\mathrm{dlog}(\lambda u^n)=n\cdot u^{-1}$, the pairing evaluates to
\[
\mathrm{Tr}_{l\otimes_{\FFp}\FFpb/\FFpb}  \mathrm{Res}\left( \sum_{j\ge 0} m' \Frob^j (\lambda_{\tau_{\alpha},\mu})u^{m'p^j-1}\cdot \lambda_{\tau_i,\mu^{-1}}u^{de_M-\xi_i'} \right ),
\]
where $\Frob:l\otimes_{\FFp} \FFpb \to l\otimes_{\FFp} \FFpb$ is induced by the $p$-th power map on $l.$ The residue is given by the coefficient of $u^{-1}$, hence this can only possibly be non-zero when $p^j m'=\xi_i'-de_M$ for some $j\ge 0.$ If $\chi=1$ we must also consider $p^j m'=\xi_{i_0}'-d_{i_0}e_M,$ but it follows easily from the observation $\xi_{i_0}'=d_{i_0}e_M$ that this can never have solutions. In the former case, if such an $j$ exists, then the pairing evaluates to
\[
m'\mathrm{Tr}_{l\otimes_{\FFp}\FFpb/\FFpb}\left(\Frob^j(\lambda_{\tau_{\alpha},\mu})\cdot \lambda_{\tau_i,\mu^{-1}}\right).
\]
Since
\[
\Frob^j(\lambda_{\tau_{\alpha},\mu})\cdot \lambda_{\tau_i,\mu^{-1}}=\Frob^j(\lambda_{\tau_{\alpha},\mu}\cdot \lambda_{\tau_{i-j},\mu^{-1}}),
\]
we find that this expression is non-zero if and only if $\tau_{\alpha}=\tau_{i-j},$ i.e. when $i_m+kf'\equiv i-j \bmod f.$

In conclusion, we have proved that for $\alpha=(m,k)\in W$ the pairing of $u_{\alpha}$ with $L_V$ is non-zero when there exist integers $0\le i<f,$ $d\in\mathcal{I}_i$ and $j\ge 0$ such that
\begin{enumerate}[\hspace{6mm} 1)]
\item $p^j m'=\xi_i'-de_M$ and
\item $i_m+kf'\equiv i-j \bmod{f}.$
\end{enumerate}
These are precisely the conditions that imply that $\alpha\in J_V^\mathrm{AH}$, as required. 
\end{proof}

\begin{cor}\label{cor:L_V-choice-indep}
The space $L_V^{\mathrm{AH}}(\chi_1,\chi_2)$ is independent of the choice of the extension $M$ of $K$ and the choice of the uniformiser $\pi$ of $M$.
\end{cor}
\begin{proof}
This follows immediately since $L_V(\chi_1,\chi_2)$ is independent of the choices and we have just proved that $L_V^{\mathrm{AH}}(\chi_1,\chi_2)=L_V(\chi_1,\chi_2)$ for any suitable set of choices.
\end{proof}

\printbibliography
 
\end{document}